\documentclass[11pt,twoside]{article}
\usepackage{amsmath,amssymb,amsthm,mathptmx,hyperref,color,nicefrac}
\usepackage[english]{babel}%
\usepackage[inner=3.3cm,outer=3.6cm,bottom=3cm]{geometry}
\usepackage{xcolor}
\usepackage{mathtools,amsfonts,amsthm}
\usepackage{enumitem}

\usepackage{color}

\theoremstyle{remark}
\newtheorem{remark}{Remark}[section]
\theoremstyle{definition}
\newtheorem{theorem}{Theorem}[section]
\newtheorem{definition}[theorem]{Definition}

\newtheorem{proposition}[theorem]{Proposition}
\newtheorem{lemma}[theorem]{Lemma}

\newtheorem{hypothesis}[theorem]{Hypothesis}

\DeclareMathOperator{\R}{\mathbb{R}}

\DeclareMathOperator{\E}{\mathcal{E}}
\DeclareMathOperator{\Y}{\mathbb{Y}}

\DeclareMathOperator{\C}{\mathcal{C}}

\DeclareMathOperator{\Se}{\mathbb{S}}

\DeclareMathOperator{\N}{\mathbb{N}}
\DeclareMathOperator{\esssup}{ess\,sup}

\DeclareMathOperator{\M}{\mathcal{M}}

\DeclareMathOperator{\cof}{cof}
\newcommand{\D}{\mathcal{D}(\mathcal{K})}

\newcommand{\Td}{\mathbb{T}^d}
\DeclareMathOperator{\ra}{\rightarrow}
\newcommand{\de}{\text{d}}
\DeclareMathOperator{\tr}{tr}

\newcommand{\sym}{\text{sym}}
\newcommand{\skw}{\text{skw}}
\newcommand{\curl}{\nabla \times }

\newcommand{\f}[1]{{\mathbf{ #1}}}
\DeclareMathOperator{\di}{\nabla\cdot}

\newcommand{\ov}[1]{\overline{{#1}}}

\renewcommand{\t}{\partial_t}

\newcommand{\sy}[1]{(\nabla \f #1)_{{\sym}}}

\newcommand{\vv}{\f \varphi}

\DeclareMathOperator{\V}{\mathbb V}

\DeclareMathOperator{\Hsig}{\f H^{1}_{0,\sigma}}

\DeclareMathOperator{\Ha}{\f L^2_{\sigma}}

\DeclareMathOperator{\He}{\f{H}^1}
\DeclareMathOperator{\Le}{\f{L}^2}
\renewcommand{\ll}[1]{\langle{#1}\rangle}

\newcommand{\intte}[1]{\int_{0}^T{ #1} \de t}
\newcommand{\inttet}[1]{\int_{0}^{t}{ #1} \de s}
\DeclareMathOperator{\BV}{{BV}([0,T])}

\newcommand{\pat}[2]{\frac{\partial #1}{\partial #2}}

\usepackage{stackengine}

\renewcommand{\di}{\nabla \cdot }

\renewcommand{\O}{\Omega}

\DeclareMathOperator{\ran}{range}
\DeclareMathOperator{\sgn}{sgn}

\allowdisplaybreaks

 \author{Robert Lasarzik%
 \thanks{Weierstrass Institute,
 Mohrenstr. 39, 10117 Berlin, Germany,
 \texttt{E-Mail: robert.lasarzik@wias-berlin.de}}}
\date{\today}
\title{Energy-variational structure in evolution equations}	
\begin{document}
\maketitle
\begin{abstract}
\noindent
We consider different measure-valued solvability concepts from the literature and show that they could be simplified by using the energy-variational structure of the underlying system of partial differential equations. In the considered examples,  we prove that a certain class of improved measure-valued solutions can be equivalently expressed as an energy-variational solution. The first concept represents the solution as a high-dimensional Young measure, whether for the second concept, only a scalar auxiliary variable is introduced and the formulation is relaxed to an energy-variational inequality. We investigate four examples: the two-phase Navier--Stokes equations, a quasilinear wave equation, a system stemming from polyconvex elasticity, and the Ericksen--Leslie equations equipped with the Oseen--Frank energy. The wide range of examples suggests that this is a recurrent feature in evolution equations in general. 
\end{abstract}
\noindent
\begin{tabular}{ll}\textbf{MSC2020:} &
35D99, 35R35, 35R45, 35Q35,   35Q74, 76T06.
\\
\textbf{Keywords:} &
Energy-variational solutions, measure-valued solutions, Varifold-solutions, \\&polyconvex elasticity, two-phase flow, Ericksen--Leslie equations
\end{tabular}

\section{Introduction}
Nonlinear evolution equations play a crucial role in modeling processes across various fields of science and technology. Their applications are far-reaching, as many natural and technical processes can be described using partial differential equations (PDEs). To gain a comprehensive understanding of these modeled phenomena and provide reliable quantitative predictions, a thorough mathematical analysis is essential.
However, this endeavor presents significant challenges due to the complex behavior of solutions to nonlinear PDEs, which often exhibit non\-smooth characteristics such as singularities or turbulence. Consequently, classical solutions are out of reach and do not even describe the observed phenomena, necessitating the use of generalized solution frameworks.
Since their first appearence in the 18th century~\cite{Lagrange}, the concept of weak solutions has become both widely accepted and powerful, especially when combined with advanced functional analytic methods~\cite{Leray}. Despite their success, weak solutions have notable limitations. For example, in the case of scalar conservation laws, the formulation does not inherently select a unique, physically relevant solution. Moreover, for multidimensional systems like the Euler and Navier–-Stokes equations, infinitely many weak solutions  exist~\cite{DeLellis, Buckmaster,Dallas}.
In addition to these inherent issues of non-uniqueness, the existence of weak solutions for many nonlinear PDE systems remains out of reach. A potential remedy, particularly for addressing the existence problem, is to explore even more generalized solution concepts. Over the years, various generalized solvability concepts have been proposed. We exemplify some of them  for the Euler equations.

The incompressible Euler equations are given by
\begin{align*}
\t \f v + (\f v \cdot \nabla) \f v + \nabla p ={} \f 0 \, ,
\quad
\di \f v ={}& 0 \quad \text {in } \Omega \times (0,T)\,,
\\
\f n \cdot \f v ={}& 0 \quad \text{on }\partial \Omega \times (0,T)\,,
\end{align*}
on a bounded Lipschitz domain $\Omega \subset \R^d$ for $d\in\N$, where the unknowns are the velocity field $ \f v : \ov\Omega \times [0,T] \to \R^d $ and the pressure $p : \ov\Omega \times [0,T] \to \R$. 
So-called dissipative solutions~\cite[Sec.~4.4]{lionsfluid}  were introduced by 
P.-L.~Lions.
Here, the weak formulation is relaxed to a relative energy inequality 
in such a way that the appearing terms are lower semicontinuous 
while retaining the  weak-strong uniqueness property. 
This property means that every generalized solution  coincides with a  classical solution that emanates from the same initial datum as long as the latter exists. 
More precisely, a function $ \f v \in \C_w([0,T];\Ha(\Omega))$,
where $\Ha(\Omega)=\{ \f u \in L^2(\Omega;\R^d) | \nabla\cdot\f u =0 \text{ in }\Omega,\,\f n \cdot \f u=0 \text{ on }\partial\Omega\}$,
is called a dissipative solution to the incompressible Euler equations, if 
\[
\frac{1}{2}\| \f v(t) -\vv(t)\|_{L^2(\Omega)}^2 + \int_0^t \left \langle  \t \vv - (\vv \cdot \nabla )\vv ,  \f v-\vv \right \rangle  e^ {\int_s^t \mathcal{K}(\vv)\de \tau }\de s 
\leq{} \frac{1}{2}\| \f v_0-\vv(0) \|_{L^2(\Omega)}^2  e^{\int_0^t \mathcal{K}(\vv)\de s } \,
\]
holds for all~$t\in(0,T)$ and all $\vv \in \C^1([0,T]; \C^1(\overline \Omega ))$ with $ \di \vv=0$ and $ \f n \cdot \vv = 0 $ on $\partial \Omega$. 
Here we set $\mathcal{K}(\vv):=\| ( \nabla \vv)_{\sym,-}\|_{L^\infty(\Omega)}$, 
where $(A)_{\sym,-}$ 
denotes the symmetric negative semidefinite part of a matrix $A\in\R^{d\times d}$.

Another prominent solution concept is based on the theory of Young measures and their generalizations. Since Young measures were first applied in the study of nonlinear PDEs~\cite{tartar}, they have become a powerful tool for identifying the limits of weakly convergent sequences under nonlinear mappings, thereby capturing oscillatory behavior of such sequences. DiPerna and Majda~\cite{DiPernaMajda} extended this framework by introducing generalized Young measures, which allowed them to account not only for oscillations but also for potential concentrations in sequences 
$\{
\f
v_
n
\}
$ of solutions to the incompressible Navier--Stokes equations with vanishing viscosity as 
$ n\to \infty$.

The corresponding limit of such solutions is identified as a so-called measure-valued solution
to the incompressible Euler equations. 
In the case of periodic boundary conditions, that is, 
when $\Omega$ is a $d$-dimensional torus $\mathbb T^d$,
then a function $\f v \in \C_w([0,T];\Ha(\Omega)		 )$ is called a measure-valued solution if 
\begin{align}
\t \f v + \di (\f v \otimes \f v + \mathfrak{R} ) +\nabla p = \f f  \,, \quad \di \f v = 0 \label{eq:meas}
\end{align}
holds in distributional sense
for a measure-valued Reynolds defect $\mathfrak{R} \in L^\infty_{w^*}(0,T;\mathcal{M}(\Omega ;\R^{d\times  d} _{\sym,+} ))$
such that
\begin{align*}
 \frac{1}{2}\left [  \int_\Omega | \f v|^2 \de \f x +\langle \mathfrak{R} , I\rangle  \right ]\Big|_0^t \leq 0 \quad\text{for a.a.~}t\in (0,T)\,.
\end{align*}
In comparison with~\cite{DiPernaMajda}, this is a more modern definition collecting all measure-related terms in the defect measure $\mathfrak{R}$ (\textit{cf.}~\cite{maxdiss}).
This notion of solution is closely related to the so-called dissipative weak solutions~\cite{Fereisl21_NoteLongTimeBehaviorDissSolEuler}.
A function  $\f v \in \C_w([0,T];\Ha(\Omega))$  is called a dissipative weak solutions if there exists a turbulence defect measure  $ \mathfrak{R} \in L^\infty_{w^*}(0,T; \mathcal{M}(\Omega ;\R^{d\times  d} _{\sym,+} )) $  such that~\eqref{eq:meas} holds 
and if there is a nonincreasing function $E\colon [0,T] \ra \R $
with
\[
\frac{1}{2}\left [  \int_{\Omega} | \f v(t)|^2 \de \f x +\langle \mathfrak{R}(t) , I\rangle  \right ] \leq E(t) 
\quad \text{ for a.a.~}t\in[0,T], \qquad
E(0-)= \frac{1}{2}\int_{\Omega} | \f v_0 |^2 \de \f x.
\]
The weak-strong uniqueness of measure-valued solutions to the incompressible Euler equations was proven in~\cite{weakstrongeuler}. 
A notable downside of measure-valued solutions is the increased degrees of freedom, as these solutions consist of several measures. To address this complexity, the concept was later simplified into so-called dissipative weak solutions, which retain only the expectation and variance of the measures in the measure-valued framework.

In a recent series of works, energy-variational solutions were introduced~\cite{envar, unsere, visco}, not only for the Euler equations but also as a more general concept~\cite{unsere, visco}. 
For the incompressible Euler equations, energy-variational solutions consist of a pair $ (\f v , E) \in \C_{w}([0,T];\Ha(\Omega)) \times \BV$ such that $E\geq \frac{1}{2}\| \f v \|_{L^2(\Omega)}^2 $ a.e.~in $(0,T)$ and
\begin{equation}
\left [ E- \int_{\Omega} \f v \cdot \vv \de x \right ] \Big|_{s-}^{t+} + \int_s^t\int_\Omega  \f v \otimes \f v : \nabla \vv  \de x + 2\| ( \nabla \vv )_{\sym,-}\|_{L^\infty(\Omega;\R^{d\times d})} \left  [ \frac{1}{2}\| \f v \|_{L^2(\Omega)}^2 - E\right ] \de \tau \leq 0 
\end{equation}
for all $ 0\leq s \leq t \leq T $ and all $\vv \in \C^1([0,T]; \C^1(\overline \Omega ))$ with $ \di \vv=0$ and $ \f n \cdot \vv = 0 $ on $\partial \Omega$. 
It was shown that energy-variational solutions are equivalent to dissipative weak solutions, while simultaneously offering a significant reduction in the degrees of freedom in the context of fluid dynamics. In~\cite{unsere}, we demonstrated the equivalence of dissipative weak solutions and energy-variational solutions for the incompressible Euler equations, incompressible magnetohydrodynamics, and the compressible Euler equations.
Moreover, energy-variational solutions were proven to exist for a broad class of systems, encompassing not only fluids but also complex fluids and viscoelastic materials. In these cases, we showed~\cite[Rem.~3.7]{visco} that, under very general conditions, the existence of energy-variational solutions implies the existence of dissipative solutions. This raises an important question: can energy-variational solutions be related to different measure-valued formulations in other contexts? For the systems studied in~\cite{visco}, there are currently no existence results for measure-valued solutions.

The aim of this paper is to revisit various systems of PDEs where generalized solvability concepts have been introduced and compare them to energy-variational solutions.
The provided examples demonstrate that the energy-variational structure arises in a wide range of contexts 
and unifies different previously defined generalized solvability concepts. 
We present an example from interface dynamics, consider a wave equation,  as well as an evolution equation stemming from polyconvex elasticity, and finally a model for liquid crystals. 
In Section~\ref{sec:1}, we focus on the two-phase Navier–-Stokes equations. The existence of varifold solutions was established in~\cite{AbelsVari}, and in a recent paper~\cite{HenselFischer}, weak-strong uniqueness of these solutions was proven. We introduce energy-variational solutions for this system in Section~\ref{sec:1} and show that they provide a stronger framework than the previously introduced varifold solutions by simulateanously reducing the degrees of freedom heavily, \textit{i.e.}, a Varifold $ V \in L^\infty_{w^*}(0,T;\mathcal{M}(\R^d \times \Se^{d-1}))$ is  replaced by $E\in \BV$.
In Section~\ref{sec:2} and Section~\ref{sec:poly}, we examine two problems in elasticity taken from~\cite{DemouliniWeakStrong}, where the authors showed the weak-strong uniqueness of measure-valued solutions. For the first problem the second order in time evolution equation for a convex elastic energy functional is considered (see Section~\ref{sec:2}) and for the second example the second order in time evolution equation for a merely polyconvex  elastic energy (see Section~\ref{sec:poly}). 
After formulating energy-variational solutions, we demonstrate that they are stronger than measure-valued solutions, proven to exist in~\cite{DemouliniExist} for the polyconvex case.
Finally, we analyze the Ericksen–Leslie equations with one simplified but still anisotropic Oseen–Frank energy in Section~\ref{sec:EL}. We compare the energy-variational solutions to newly introduced dissipative weak solutions and, subsequently, to the measure-valued solutions proposed in~\cite{meas}, where weak-strong uniqueness was also established in~\cite{measweakstrong}.

Given the broad range of examples, energy-variational solutions can be regarded as a highly general framework applicable to many different PDE systems. Their key advantage lies not only in the reduction of degrees of freedom but also in the existence of a general time discretization scheme, analogous to the minimizing movement scheme for gradient flows, which facilitates the approximation of this large class of systems.
Additionally, energy-variational solutions are well-suited for the introduction of selection criteria. This approach was explored in~\cite{envar} for incompressible fluid dynamics, where various selection criteria were proposed. In particular, the maximal dissipation criterion, as formulated by Dafermos~\cite{dafermos2}, can be naturally adapted to energy-variational solutions. Specifically, the auxiliary variable 
$E$ should be minimized to ensure that its distance from
$E(\f U) $ is as small as possible~\cite{FeireislNew}. 

\section{Notation and preliminaries}
The standard matrix and matrix-vector multiplication is written without an extra symbol,
$$\f A \f B =\left [ \sum _{j=1}^d \f A_{ij}\f B_{jk} \right ]_{i,k=1}^d \,, \quad  \f A \f a = \left [ \sum _{j=1}^d \f A_{ij}\f a_j \right ]_{i=1}^d\, , \quad  \f A \in \R^{d\times d},\,\f B \in \R^{d\times d} ,\, \f a \in \R^d .$$
The outer vector product is given by
 $\f a \otimes \f b := \f a \f b^T = \left [ \f a_i  \f b_j\right ]_{i,j=1}^d$ for two vectors $\f a , \f b \in \R^d$. 
The symmetric and skew-symmetric parts of a matrix are given by 
$\f A_{\sym}: = \frac{1}{2} (\f A + \f A^T)$ and 
$\f A _{\skw} : = \frac{1}{2}( \f A - \f A^T)$, respectively ($\f A \in \R^{d\times  d}$).
The identity matrix is denoted by $ \mathbb I$ and the zero matrix by $\mathbb O$. 
We denote $ |\f a |_{\varepsilon(\f d)}^2 = |\f a|^2 + \varepsilon ( \f d \cdot \f a )^2 $ for $ \f a$, $\f d \in \R^d$.
The symmetric, symmetric positive definite, and symmetric trace-less  matrices, we denote by $\R^{d\times d}_{\sym}$, $\R^{d\times d}_{\sym,+}$, $\R^{d\times d}_{\sym,0}$, respectively. 
By $ \Td$ we denote the $d$-dimensional flat torus and $\Se^d$ the unit sphere in $d$ dimensions. 
We use  the Nabla symbol $\nabla $  for real-valued functions $f : \R^d \to \R$, vector-valued functions $ \f f : \R^d \to \R^d$ as well as matrix-valued functions $\f A : \R^d \to \R^{d\times d}$ denoting
\begin{align*}
\nabla f := \left [ \pat{f}{\f x_i} \right ] _{i=1}^d\, ,\quad
\nabla \f f  := \left [ \pat{\f f _i}{ \f x_j} \right ] _{i,j=1}^d \, ,\quad
\nabla \f A  := \left [ \pat{\f A _{ij}}{ \f x_k} \right ] _{i,j,k=1}^d\, .
\end{align*}
 The divergence of a vector-valued and a matrix-valued function is defined by
\begin{align*}
\di \f f := \sum_{i=1}^d \pat{\f f _i}{\f x_i} = \tr ( \nabla \f f)\, , \quad  \di \f A := \left [\sum_{j=1}^d \pat{\f A_{ij}}{\f x_j}\right] _{i=1}^d\, .
\end{align*}

We also use the Levi--Civita tensor $\f \epsilon$. Let $\mathfrak{S}_3$ be the symmetric group of all permutations of $(1,2,3)$. The sign of  a given permutation  $\sigma \in \mathfrak{S}_3$ is denoted by $\sgn \sigma $.
The tensor~$\f \epsilon$ is defined via
\begin{align}\label{levi}
\f \epsilon_{ijk}:= \begin{cases}
\sgn{\sigma},  &  ( i ,j,k) = \sigma( 1,2,3)\text{ with } \sigma\in \mathfrak{S}_3 ,\\ 
0, & \text{ else}\, .
\end{cases} 
\end{align}
This tensor allows  to write the cross product via 
\begin{align*}
(\f a \times \f b)_i = \left (\f \epsilon :(\f a \otimes \f b)\right )_i =\sum_{j,k=1}^d\f \epsilon _{ijk} \f a_j \f b _k\, \quad \text{for all }\f a , \f b \in \R^d \, 
\end{align*}
and the curl via
\begin{align*}
(\curl \f d)_i = \sum_{j,k=1}^d \f\epsilon_{ijk} \partial_j \f d_k \, \quad\text{for all }\f d \in \C^1 ( \Omega)\, .
\end{align*}
We also use the cross product for a vector matrix product, \textit{i.e.,}
$ (\f a \times \f A )_{ij}= \sum_{k,l=1}^d\f \epsilon_{ikl} \f a_k \f A_{lj}$ for $\f a\in\R^d$ and $\f A\in\R^{d\times d}$ as well as
$ (\nabla\times \f A )_{ij}= \sum_{k,l=1}^d\f \epsilon_{jkl} \partial_{x_k} \f A_{il}$ for $\f A\in\C^1(\Omega;\R^{d\times d})$.

The space of smooth solenoidal functions with compact support is denoted by $\mathcal{C}_{c,\sigma}^\infty(\Omega;\R^d)$. By $\f L^p_{\sigma}( \Omega) $, $\Hsig(\Omega)$,  and $ \f W^{1,p}_{0,\sigma}( \Omega)$, we denote the closure of $\mathcal{C}_{c,\sigma}^\infty(\Omega;\R^d)$ with respect to the norm of $ L^p(\Omega;\R^d) $, $  H^1( \Omega;\R^d) $, and $  W^{1,p}(\Omega;\R^d)$, respectively.
The dual space of a Banach space $\V$ is always denoted by $ \V^*$ and equipped with the standard norm; the duality pairing is denoted by $\langle\cdot, \cdot \rangle$.  We denote the conjugated exponent of  $ p\in [1,\infty]$  by $p'$, \textit{i.e.,} $ p' = \frac{p}{p-1}$. 
We denote the continuous functions vanishing on the closure of the domain by $\C_0(\Omega)$ and $ \mathcal{M}(\Omega)$ denote the Radon measures over $\Omega$, where we note that $ \mathcal{M}(\Omega) = ( \C_0(\Omega))^*$ and $\mathcal{M}(\ov\Omega) = ( \C(\ov\Omega))^*$, since $\Omega $ is an open bounded domain (\textit{cf}.~\textit{e.g.}~\cite{Roubicekmeas}).  The space of probability measures is denoted by $ \mathcal{P}(\Omega)$ and the space of nonnegative measures by $\mathcal{M}^+(\Omega)$ for a certain set $\Omega$. 
The total variation of a function $E:[0,T]\ra \R$ is given by 
$$ | E |_{\text{TV}([0,T])}= \sup_{0=t_0<\ldots <t_n=T} \sum_{k=1}^n \lvert E(t_{k-1})-E(t_k) \rvert\,, $$
where the supremum is taken over all finite partitions of the interval $[0,T]$. 
We denote the space of all integrable functions on $[0,T]$ 
with bounded variation by~$\BV$, and we equip this space with the norm
$\| E \|_{\BV} :=  \| E\|_{L^1(0,T)} + | E |_{\text{TV}([0,T])}$~(\textit{cf}.~\cite{BV}). 
Recall that an integrable function $E$ has bounded variation if and only if its
distributional derivative $E'$ is an element of 
$\mathcal M([0,T])$, the space of finite Radon measures on $[0,T]$.
Moreover, $\BV$ coincides with the dual space of a Banach space,
see \cite[Remark~3.12]{AmbrosioFusoPallara_BVFunctions_2000} for example.
For a Banach space $ \V$, Bochner--Lebesgue spaces are denoted  by $ L^p(0,T; \V)$ for $p\in[1,\infty]$ and $L^p_{w^*}(0,T; \V^*)$ for strongly and weakly measurable functions, respectively. Moreover,  $W^{1,p}(0,T;\V)$ denotes the Banach space of abstract functions in $ L^p(0,T; \V)$ whose weak time derivative exists and is again in $ L^p(0,T; \V)$ (see also
\cite[Section~II.2]{DiestelUhl} or
\cite[Section~1.5]{roubicek} for more details).
By  $\C([0,T]; \V) $ and $ \C_w([0,T]; \V)$, we denote the spaces of abstract functions mapping $[0,T]$ into $\V$ that are  continuous  with respect to the strong and weak topology in $\V$, respectively.
\begin{lemma}\label{lem:hahn}
Let  $ \f l : \mathcal{V} \ra \R$ be a linear continuous functional,
where $ \mathcal{V} $ is a closed subspace of 
\[
\mathcal U:= \left \{\f \varphi\in L^1(0,T; W^{1,1}(\Td; \R^m ))\mid  \int_\O\f\varphi\de x =0\right \}
\]
with $ d$, $m\in \N$ and let $\mathcal W $ be a Banach space such that $\nabla \mathcal V\subset  \mathcal{W}$, where 
\[
\nabla \mathcal V:= \left \{\nabla \f \varphi\mid \f \varphi \in 
\mathcal V \right \}\,.
\]
Let
$\mathfrak p : L^1(0,T;\mathcal{W})  \ra \R$ be a sublinear mapping such that 
\begin{equation}
    \forall \f \psi\in \mathcal V \colon \quad \langle \f l , \f \psi \rangle \leq \mathfrak p ( \nabla \f\psi  ) \,.\label{est:l}
\end{equation}
Then there exists
 an element
$$\mathfrak R \in (L^1(0,T; \mathcal{W}
))^*=L^\infty_{w^*}(0,T;\mathcal{W}^*) $$ 
satisfying 
\[
\forall\Phi\in L^1(0,T; \mathcal{W}):\ \langle -\mathfrak R, \Phi\rangle \leq \mathfrak p(\Phi),
\qquad
\forall\f \psi\in \mathcal{V}: \ 
\langle -\mathfrak R, \nabla \f\psi \rangle = \langle \f l, \f \psi\rangle.
\]
Moreover, if $ \mathcal W^* $ is reflexive and separable, we have $\mathfrak R \in (L^1(0,T; \mathcal{W}
))^*=L^\infty(0,T;\mathcal{W}^*) $.
\end{lemma}

\begin{proof}
First consider $\f \psi\in\mathcal V$ with $\nabla \f \psi =0$.
This implies that $\f \psi(\cdot,t)$ is affine linear,
and since $\f \psi\in\mathcal V$ is spatially periodic and has vanishing mean value,
this is only possible for $\f \psi=0$. 
We can define the functional $L$ by $\langle L,\Psi\rangle=\langle \f l,\f \psi\rangle$
for $\Psi=\nabla \f \psi \in\mathcal W$.
Then estimate \eqref{est:l} implies
\begin{equation}\label{linearformest}
\langle L, \Psi \rangle \leq \mathfrak p(\Psi)
\end{equation}
for all $\Psi = \nabla \f \psi \in  L^1(0,T;\mathcal W)$.
By the Hahn--Banach theorem (see \textit{e.g.}~\cite[Thm~1.1]{brezis}), we may extend $L$ 
 to a linear functional on $L^1(0,T;\mathcal{W})$. 
Using the Riesz representation theorem, 
we may identify this extension with an object $-\mathfrak R$
such that the asserted properties are satisfied.

The Theorem by Pettis~\cite{DiestelUhl} implies that weak measurability is equivalent to strong measurability for separable spaces, since $\mathcal W$ is also reflexive  the last statement follows.
\end{proof}

\begin{lemma}[Construction of probability measure]\label{lem:Young}
Let $\eta : \R^n \to \R$   be a strictly convex superlinear and continuous function with $\eta(0)=0$ and $\eta(\f y) \geq 0$ for all $ \f y \in \R^n$.  Let $\Omega \subset  \R^d$ be a domain. 
We define $ V := \{ u\in  L^1(\Omega ; \R^n) \mid \int_{\Omega} \eta( u )\de x <\infty\}$ and observe that $ V $ is a closed subset of $  L^1(\Omega ; \R^n)$ and as such a Banach space itself. 
Let $ g : \R^n \to \R^m $ be a  continuous  function that is not linear such that $\lim_{|\f y| \to \infty}\frac{|g(\f y)|}{\eta(\f y)}= 0$. 
We assume that there exists a reflexive Banach space $W\subset  L^1(\Omega ; \R^m)$ equipped with a norm $ \|\cdot \|_{W}$ such that  the Nemytskii mapping induced by $g: V \to W $ is surjective and 
\begin{equation}\label{bound}
 \int_{\Omega}  g( \f y(x) ) \cdot  A (x) \de x \leq \| A\|_{W}\left [ \int_\Omega \eta(\f y(x)) \de x+c \right ] \text{ for all } y\in V \text{ and } A \in  W\,. 
\end{equation}

Then it holds that for every $ \f y_0\in L^\infty (0,T;V) $, $z_0 \in L^\infty(0,T; W^*)$, and $ x_0\in \C([0,T])$ with $x_0\geq 0$~in $[0,T]$
   there exists a $ \nu \in L^\infty_{w^*}(0,T; L^1(\Omega; \mathcal{P}(\R^n)))$ and a $ \gamma \in L^\infty_{w^*} (0,T; \mathcal{M}(\Omega))$ with $ \gamma \geq 0$ 
   such that
\begin{subequations}
\begin{align}\label{moments}
\f y_0 = \int_{\R^n} \f s \nu (\de \f s) \,, \quad   g (\f y_0) + \f z _0 &= \int_{\R^n} g (\f s ) \nu (\de \f s)\,,
\quad\text{a.e.~in }\Omega \times (0,T) 
 \quad \text{and } \\
 \int_{\Omega}\eta(\f y_0 ) \de x + \| \f z_0\|_{W^*}  + x_0   &=  \int_\Omega
\int_{\R^n} \eta(\f s) \de \nu(\f s)\de x+ \int_{\Omega} \de \gamma_t(x) \quad \text {a.e.~in }(0,T)
 \,.\label{moments2}
\end{align} 
\end{subequations}
\end{lemma}
\begin{proof}
First, we note that from~\eqref{bound} we infer by duality that $ \| g(\f y (x))\|_{W^*} \leq \int_\Omega \eta(\f y(x)) \de x+ c $ for all $ \f y \in V$. Due to the surjectivity of $g$, we observe that there exists a $\f y_1\in L^\infty(0,T;V) $ such that $ g(\f y_1(x,t))=\f z_0(x,t)$. Note that the measurability of $\f y_1$ follows from the measurability of $\f z_0$ and the continuity of $g$.  
We want to consider a compactification of 
 $\R^n$.   
Therefore, we define the bijection $ \Phi: \R^n \to B^n$ via $\Phi(\f y) = \frac{\f y}{1+\eta(\f y)}$. Here $ B^d$ denote the usual open unit-ball in $\R^d$. 
Via the transformation $ \tilde{f}(x,\f s) : = \frac{f (x,\Phi^{-1}( \f s)) }{1+\eta( \Phi^{-1}(\f s))}$ we define $\tilde{f}\in \C(\ov\Omega\times B^n)$ for every $ f \in \C(\ov\Omega\times \R^n)$. We consider the functions 
\[
\C_{\eta}(\ov\Omega\times\R^n) :=\left \{f \in \C(\ov\Omega \times \R^n) \mid \exists g \in \C(\ov\Omega \times\ov{ B^n}) \text{ such that } g(x,\f s ) = \tilde{f}( x, \f s) \text{ on }\ov\Omega \times B^n\right \}\,. 
\]
It is an elementary exercise to show that this is indeed a Banach space if equipped with the norm $$ \| u\|_{\C_{\eta}(\ov\Omega \times \R^n)} : =\| \tilde{u}\|_{\C(\ov\Omega\times \ov{B^n})}= \sup_{(x,\f s)\in \ov\Omega\times \ov{B^n}} | \tilde{u}(x,\f s) | =  \sup_{x\in\ov\Omega} \sup_{\f y\in \R^n} \frac{|u(x, \f y) |}{1+\eta(\f y)}\,.$$
%
%
This unit Ball can now be compactified in the usual way, such that we    
   observe that
$$ X 
:= \C( \ov \Omega\times\ov{B^n} )  \text{  such that }
   (\C(\ov\Omega \times \ov{B^n}))^* = \mathcal{M}(\ov \Omega \times\ov{B^n})\,,$$ since $\ov \Omega \times \ov{B^n}$ is compact~\cite[Thm.~1.32~(iii)]{Roubicekmeas}.

As a linear subspace of $L^1(0,T;X)$, we consider the space 
\[Y :=  \left  \{ u \in L^1(0,T;X) \Bigg| \begin{matrix}
\exists ( \alpha_0 , \f \alpha _1 , \f \alpha _2, \alpha_3)\in L^1(0,T;
\C(\ov\Omega) \times \C(\ov\Omega;\R^n)\times  \C(\ov\Omega;\R^m) \times \R )
\\
u(x ,t,  \f s ) = \alpha_0(x,t ) + \f \alpha_1(x,t ) \cdot \f s + \f \alpha_2(x,t ) \cdot g(\f s)  +\alpha _3(t) \eta(\f s) 
\end{matrix} \right \} \,,\]
which is a linear subspace of $L^1(0,T;X) $. 
This holds due to the fact that $\eta$ is super linear and $g$ is not linear and also linearly independent of $\eta$. 
On $Y$, we define the linear functional $ \f L : Y \to \R$ via 
$$
\begin{aligned}
 \langle \f L, u\rangle &=\int_0^T \int_{\Omega} \alpha_0(x,t)  + \f \alpha_1(x,t) \cdot \f y_0(x,t) + \f \alpha_2(x,t) \cdot (g(\f y_0(x,t))+ g(\f y_1(x,t)) )\de x\de t  \\&\quad +\int_0^T \alpha_3(t) \int_{\Omega} \eta(\f y_0(x,t))\de x  + \alpha_3(t) \| g(\f y_1(t))\|_{W^*}  + \alpha_3 (t)
 x_0(t)
 \de t  \,.
\end{aligned} 
 $$
This is a linear injective functional on the linear subspace $Y$ of $L^1(0,T;X)$.
We need to prove that $\f L$ is continuous on $L^1(0,T;X)$. 
Therefore, we consider
\begin{align*}
{\langle \f L , u \rangle } &=\int_0^T\int_{\Omega} \alpha_0(x,t) + \f \alpha_1(x,t) \cdot \f y_0(x,t) + \f \alpha_2 (x,t)\cdot (g(\f y_0(x,t))+ g(\f y_1(x,t)) ) \de x\de t 
 \\&\quad +\int_0^T \alpha_3(t)\left ( \int_\Omega  \eta(\f y_0(x,t))  
\de x + x_0(t) +  \|g(\f y_1(t))\|_{W^*}\right )\de t 
 \\&= \int_0^T \int_\Omega \alpha_0(x,t) + \f \alpha_1(x,t) \cdot \f y_0(x,t) + \f \alpha_2(x,t) \cdot g(\f y_0(x,t))  + \alpha_3 (t)  \eta(\f y_0(x,t))\de x \de t \\ &\quad  +\int_0^T \int_{\Omega}  \alpha_0(x,t) + \f \alpha_1(x,t) \cdot \f y_1(x,t) + \f \alpha_2 (x,t)\cdot g(\f y_1(x,t))\de x  +\alpha_3(t) \|g(\f y_1(t))\|_{W^*})   \de t 
\\  &\quad -\int_0^T\int_{\Omega}  \alpha_0(x,t) + \f \alpha_1(x,t) \cdot \f y_1(x,t)\de x 
 + \alpha_3 (t) x_0(t)
 \de t 
\\
&\leq \int_0^T\int_{\Omega} \frac{u(x,t, \f y_0(x,t))  }{1 + \eta(\f y_0(x,t))} ( 1 + \eta(\f y_0(x,t)))  +\frac{u(x, \f y_1(x,t))  }{1 + \eta(\f y_1(x,t))} ( 1 + \eta(\f y_1 (x,t)))  \de x \de t 
\\
&\quad+ |\Omega| \| \alpha_0\|_{L^1(0,T;\C(\ov\Omega))} + \| \f \alpha_1\|_{L^1(0,T;\C(\ov\Omega ; \R^n))} \| \f y_1\|_{L^\infty(0,T;L^1(\Omega))} \\&\quad + c  + \|\alpha_3\|_{L^1(0,T)} \esssup_{t\in(0,T)}|x_0(t)|
\\
&\leq  c 
\| u \|_{L^1(0,T;X)} \left(1 + \esssup_{t\in(0,T)}\left (\int_{\Omega} \eta(\f y_0) +\eta(\f y_1) 
\de x+ |x_0|\right )  \right) < \infty \,.
\end{align*}
The first inequality follows from the inequality $ \|g(\f y_1)\|_{W^*} \leq \int_{\Omega} \eta(\f y_1) \de x + c$ for all $\f y _1 \in V$.
The last inequality follows from $\frac{u(x, \f s )  }{1 + \eta(\f s)} \leq \| u \|_{X} $ for all $ \f s \in \R^n$. 
The superlinearity of $ \eta$ guarantees that $\| \f y_1\|_{L^1(\Omega)}$ is bounded.  
We observe by the conditions on $\eta$ and $g$  that \[
 |\alpha_3| =  \lim_{|\f y|\to\infty} \frac{|  u(x,\f y)|}{1+\eta(\f y)} \leq \|u\|_{X}\,.
\]
Note that also the norms of the coefficient functions $ \alpha_0$ and $\f \alpha_1$ can be bounded accordingly. 

This implies that $\f L$
  can be extended by Hahn--Banach~\cite[Cor.~1.2 ]{brezis} to a  functional $\f l$ on $L^1(0,T;X)$, \textit{i.e.,} $ \f l \in (L^1(0,T;X))^* = L^\infty_{w^*}(0,T;X^*)$  such that it coincides with the functional $\f L$ on $ Y$ \textit{i.e.,} $ \langle \f l  , u \rangle = \langle \f L , u \rangle $ for all $u\in Y$ and $ \| \f l  \|_{X^*} = \| \f L\|_{Y^*}$.

 We may express the linear mapping $\f l \in L^\infty_{w^*}(0,T;\mathcal{M}(\ov \Omega \times \ov{ B^n}))$ via  
the measures, $ \f l (x,t, \f s ) )=  \mu _t(x, \f s )$   with  $ \mu_t \in \mathcal{M}(\ov \Omega\times \ov{ B^n})$ for a.e.~$t\in(0,T)$.
We define the recession  function $f^\infty\in\C(\ov\Omega\times \Se^{n-1})$ for a function $f\in \C_{\eta}(\ov\Omega\times \R^n)$ via 
$f^\infty(x,\f s) = \lim _{r\to\infty} \frac{f(x,r\f s) }{1+\eta(r\f s)}$. 
We define $\tilde{\mu}_t \in \mathcal{M}(\ov\Omega\times \R^n) $ via 
\begin{equation}
\int_{\ov\Omega} \int_{\R^n} f(x,\f y ) \de \tilde{\mu}_t(x,\f y ) := \int_{\ov\Omega}\int_{{B^n}} \tilde{f}(x,\f s) \de \mu_t(x,\f s) \quad \text{for all }f\in \C_{\eta}(\ov\Omega\times \R^n)\,.
\end{equation}

We may identify the measures $\mu_t$ with the sum of two measures, a measure on $\ov \Omega \times \R^n$ and a  measure on  $\ov \Omega \times \Se^{n-1}$ via
\begin{align*}
\int_{\ov\Omega}\int_{\ov{B^n}} \tilde{f}(x,\f s) \de \mu_t(x,\f s) 
&= \int_{\ov\Omega}\int_{{B^n}} \tilde{f}(x,\f s) \de \mu_t(x,\f s) +\int_{\ov\Omega}\int_{\Se^{n-1}} \tilde{f}(x,\f s) \de \mu_t(x,\f s) 
\\
&= \int_{\ov\Omega} \int_{B^n} \frac{f(x,\Phi^{-1}(\f s )) }{1+\eta(\Phi^{-1}(
\f s ))} \de \mu _t(x,\f s) +\int_{\ov\Omega}\int_{\Se^{n-1}} \tilde{f}(x,\f s) \de \mu_t(x,\f s) 
\\
&= \int_{\ov\Omega} \int_{\R^n} f(x,\f y ) \de \tilde{\mu}_t(x,\f y ) +\int_{\ov\Omega}\int_{\Se^{n-1}} {f}^\infty (x,\f s) \de \mu_t(x,\f s) 
\end{align*}
For $\tilde{\mu}_t$, we  apply desintegration of measures and afterwards the Radon--Nikodym derivative with respect to the Lebesgue-measure on $\Omega$, such that there exists $ \nu
 \in L^\infty_{w^*} (0,T; L^1(\Omega;\mathcal{M}(\R^n))) $ and $ \omega
  \in L^\infty_{w^*}(0,T; \mathcal{M}(\ov \Omega\times{\R^n}))$ with $ \de  \tilde{\mu}_t (x, \f s) = \de \nu_{x,t} (\f s)\de x + \de \omega_{t}(x, \f s )$, where we write the singular measure again as a measure on $\ov\Omega\times {\R^n}$. Note that $\omega $ is singular with respect to the  Lebesgue measure on $\ov\Omega$ such that it is only nonzero on  sets of Lebesgue-measure zero in $\Omega$. 
  
We observe that the recession functions $\tilde{1}$, $\tilde{\mathrm{id}}$, and $\tilde{g}$ to the functions $ 1$, $\mathrm{id}$, and $g$ are zero, since $\eta$ is superlinear and $ \lim_{|\f y|\to\infty} \frac{g(\f y)}{1+\eta(\f y)}=0$. Here $\mathrm{id}$ denotes the identity on $\R^n$. 
  Choosing first $\alpha_0\neq 0$ and $\alpha_i = 0$ for all $ i> 0$, secondly $ \alpha_1\neq 0$ and $\alpha_i = 0$ for all $ i\neq 1 $, and thirdly $ \alpha_2 \neq 0$ and $\alpha_i = 0$ for all $ i\neq 2 $, we observe
  \begin{align*}
  \int_0^T\int_\Omega \alpha_0(x,t) \de x\de t  &= \int_0^T\int_{\ov\Omega} \int_{{\R^n}} \alpha_0(x,t) \de (\nu_{x,t}(\f s) \de x+  \omega_{t}(x,\f s))\de t 
  \,, \\
  \int_0^T\int_{\Omega} \alpha_1(x,t) \f y_0(x) \de x\de t &= \int_0^T \int_{\ov\Omega} \int_{{\R^n}}  \alpha_1(x,t)\cdot  \f s  \de (\nu_{x,t}(\f s) \de x+ \omega_{t}(x,\f s))\de t 
  \,,
  \\
  \int_0^T\int_\Omega \alpha _3(x,t)\cdot  ( g( \f y_0(x,t) + g( \f y_1(x,t)) \de x\de t &= \int_0^T\int_{\ov\Omega} \int_{{\R^n}} \alpha _3(x,t) \cdot g(\f s) \de (\nu_{x,t}(\f s) \de x+ \omega_{t}(x,\f s))\de t
  \, 
  \end{align*}
  for all $\alpha_0 \in L^1(0,T;\C(\ov\Omega))$, $\alpha_1 \in L^1(0,T; \C(\ov\Omega;\R^n))$, and $\alpha_2 \in L^1(0,T; \C(\ov \Omega;\R^m))$. 
  In all those case, we infer that  the singular measure $ \omega_t$ vanishes, \textit{i.e.,}
$\int_{\ov\Omega} \phi(x) \int_{\ov{\R^{n}}} (1+\f s + g(\f s) )   \de \omega _{t} (x,\f s ) =0 $ for all $ \phi \in \C(\ov\Omega)$. 
Additionally, we infer that $ \int_{{\R^n}} \de \nu_{x,t} (\f s) = 1$,   
  which implies   $ \nu_{x,t} \in \mathcal{P}(\R^n)$ for a.e.~in $\Omega\times (0,T)$.
  Moreover, the above equations imply~\eqref{moments}.
  Finally, choosing $ \alpha_0=\alpha_1=\alpha_2=0$, we find
  \begin{multline*}
   \left (\int_0^T\alpha _3(t) \int_\Omega \eta(\f y_0(x,t))
    \de x+x_0(t) + \|g(\f y_1)\|_{W^*} \de t \right ) 
   \\=\int_0^T \alpha _3(t)\left ( \int_{\Omega}\int_{{\R^n}} \eta(\f s) (\de \nu_x(\f s) \de x + \de \omega_{t}(x,\f s) ) +\int_{\ov\Omega}\int_{\Se^{n-1}}  \de \mu_t(x,\f s) 
\right ) \de t 
  \end{multline*}
   for all $ \alpha_3 \in L^1(0,T)$. Note that the recession function to $\eta$ is $1$. 
  Defining   $$ \int_{\ov\Omega} \phi (x) \de \gamma_t(x) := \int_{\ov\Omega} \int_{{\R^n}} \eta(\f s)\phi(x)  \de \omega_{t}(x,\f s)+\int_{\ov\Omega}\int_{\Se^{n-1}} \phi(x) \de \mu_t(x,\f s) \text{  for all  }\phi \in \C(\Omega)\,,$$ we infer that $ \gamma_t\in \mathcal{M}^+(\Omega)$ and~\eqref{moments2}.
 \end{proof}

\begin{lemma}[Factorization of linear maps]\label{lem:fact}
Let $ \mathbb X$ and $\mathbb Y$ be two normed spaces and $ \f l : \mathbb X \to \R$ be a linear functional and $ A : \mathbb X \to \mathbb Y $ a linear functional such that 
 $ \langle \f l  , x \rangle_X \leq p( A( x ) )$ with $ p: \mathbb Y \to [0,\infty)$ being a Minkowski functional, \textit{i.e.,} $p(\lambda y)= \lambda p(y)$ for all $\lambda \geq   0$, $y\in \mathbb Y$ and $ p(y_1+y_2)\leq p(y_1)+p(y_2)$ for all $y_1$, $y_2\in \mathbb Y$.  
Then there exists a linear functional $ \f R : \mathbb Y \to \R$ such that $ \langle \f R , A ( x)\rangle_Y = \langle \f l ,x \rangle_X $ for all $x\in \mathbb X$
 as well as $  \langle \f R , y\rangle_Y \leq p( y)$ 
for all $y\in \mathbb Y$. 
\end{lemma}
\begin{proof}
First, we observe that $ \ker \f l \subset \ker A $.
Indeed, from $ A(x)=0$, we infer due to the linearity that $ \alpha \langle \f l , x\rangle \leq | \alpha| p( A( x) )= 0$ for all $\alpha \in \R$ and  $x \in \mathbb X$, which implies $ \langle \f l , x \rangle = 0 $. 
Hence, we have that $A$ is a bijective mapping on the quotient space  $ \tilde A : \mathbb X _{\backslash \ker A} \to \ran A $, where $\tilde A (x ) = A (x)$ for all $x\in \mathbb X$. 

We define the mapping $ \f L : \ran A \subset W \to  \R $ via $\langle  L , y\rangle _{\mathbb Y}  = \langle \f l , (\tilde A)^{-1} (y) \rangle_{\mathbb X}\leq p(y) $ for all $ y \in \ran A$. The Theorem of Hahn--Banach now guarantees that this mapping $\f L$ can be extended to a mapping on $\mathbb Y$ such that $ \f R : \mathbb Y \to \R$ has the asserted properties. 
\end{proof}

\section{Interface evolution in two phase Navier--Stokes equations}\label{sec:1}
We consider the flow of two incompressible fluids with surface tension on the surface between the two fluids. 
The volume occupied by the first fluid is denoted by the indicator function $\chi=\chi(x, t)$.
The local fluid velocity is denoted by  $\f v=\f v(x, t)$, and the local pressure by $p=p(x, t)$. The fluid-fluid interface moves just according to the fluid velocity, where the evolution of the velocity of each fluid and the pressure are determined by the incompressible Navier--Stokes equations.  Moreover,  the fluid-fluid interface exerts a surface tension onto the fluids proportional to the mean curvature of the interface. Together with the natural no-slip boundary condition and the appropriate boundary conditions for the stress tensor on the fluid-fluid interface, the system is given by 
\[
\begin{aligned}
\partial_t \chi+(\f v \cdot \nabla) \chi&{}=0&& \text{in }\R^d \times (0,T) \,,\\
\rho(\chi) \partial_t \f v+\rho(\chi)(\f v \cdot \nabla) \f v&{}=-\nabla p+\nabla \cdot\left(2\mu(\chi)(\nabla \f v )_{\sym}\right)+\sigma \mathrm{H}|\nabla \chi|&& \text{in }\R^d\times (0,T)\,,\\
\nabla \cdot \f v&{}=0&& \text{in }\R^d\times (0,T)\,,\\
(\chi(0), \f v (0)) &{}= (\chi_0,\f v_0)&& \text{in }\R^d\,.
\end{aligned}\]
where H denotes the mean curvature vector of the interface $\partial\{\chi=0\}$ and $|\nabla \chi|$ denotes the surface measure $\left.\mathcal{H}^{d-1}\right|_{\partial\{\chi=0\}}$. 
 For the sake of brevity, we have used the abbreviations $\rho(\chi):=\rho^{+} \chi+\rho^{-}(1-\chi)$ and $\mu(\chi):=\mu^{+} \chi+\mu^{-}(1-\chi)$. 
Here, $\mu^+$ and $\mu^-$ are the shear viscosities of the two fluids and $\rho^+$ and $\rho^-$ are the densities of the two fluids. The constant $\sigma$ is the surface tension coefficient. 
Let a surface tension constant $\sigma>0$, the densities and shear viscosities of the two fluids $\rho^{ \pm}, \mu^{ \pm}>0$, a finite time $T >0$, a solenoidal initial velocity profile $v_0 \in L^2\left(\mathbb{R}^d ; \mathbb{R}^d\right)$, and an indicator function of the volume occupied initially by the first fluid $\chi_0 \in \mathrm{BV}\left(\mathbb{R}^d\right)$ be given.

The total energy of the system is given by the sum of kinetic and surface tension energies
$$
\E(\chi, v):=\int_{\mathbb{R}^d} \frac{1}{2} \rho(\chi)|\f v|^2 \mathrm{~d} x+\sigma \int_{\mathbb{R}^d} 1 \mathrm{~d}|\nabla \chi|\,.
$$

It is at least formally subject to the energy dissipation inequality
$$
\E(\chi, \f v)(T)+\int_0^T \int_{\mathbb{R}^d} 2{\mu(\chi)}\left|(\nabla \f v )_{\sym}\right|^2 \mathrm{~d} x \leq \E(\chi, \f v   )(0) .
$$

\begin{definition}[Varifold solution for the two-phase Navier--Stokes equation] \label{def:vari}
A triple $(\chi, \f v, V)$ consisting of a velocity field $\f  v$, an indicator function $\chi$ of the volume occupied by the first fluid, and an oriented varifold $V$ with
$$
\begin{aligned}
\f v & \in L^2(0, T ; \Hsig (\R^d) \cap L^{\infty}(0, T  ; \Ha(\mathbb{R}^d)), \\
\chi & \in L^{\infty}(0, T  ; \mathrm{BV}(\R^d  ;\{0,1\})), \\
V & \in L_{w^*}^{\infty}(0, T  ; \mathcal{M}(\R^d \times \mathbb{S}^{d-1})),
\end{aligned}
$$
is called a varifold solution to the free boundary problem for the Navier--Stokes equation for two fluids with initial data $\left(\chi_0,\f v _0\right)$ if the following conditions are satisfied:
\begin{itemize}
\item The velocity field $\f v$  has vanishing divergence $\nabla \cdot\f v =0$ and the equation for the momentum balance
\begin{equation}\label{TwoPhaseNav}
\begin{aligned}
& \int_{\mathbb{R}^d} \rho(\chi(\cdot, t))\f v (\cdot, t) \cdot \f \eta(\cdot, t) \mathrm{d} x-\int_{\mathbb{R}^d} \rho\left(\chi_0\right)\f v _0 \cdot \f \eta(\cdot, 0) \mathrm{d} x \\
&= \int_0^t \int_{\mathbb{R}^d} \rho(\chi) \f v \cdot \partial_t \f \eta \mathrm{~d} x \mathrm{~d} t+\int_0^t \int_{\mathbb{R}^d} \rho(\chi) \f v  \otimes\f v : \nabla \f \eta \mathrm{d} x \mathrm{~d} t \\
& \quad-\int_0^t \int_{\mathbb{R}^d} \mu(\chi)\left(\nabla\f v \right)_{\sym}: \nabla \f \eta \mathrm{d} x 
-\sigma  \int_{\mathbb{R}^d \times \mathbb{S}^{d-1}}(\operatorname{Id}-s \otimes s): \nabla \f \eta \mathrm{d} V_t(x, s) \mathrm{d} t
\end{aligned}
\end{equation}
is satisfied for a.e.~$t \in\left[0, T \right)$ and every smooth vector field $\f \eta \in$ $C _{\mathrm{c}}^{\infty}\left(\mathbb{R}^d \times\left[0, T \right) ; \mathbb{R}^d\right)$ with $\nabla \cdot \f \eta=0$. 
\item The indicator function $\chi$ of the volume occupied by the first fluid satisfies the weak formulation of the transport equation
\begin{equation}
\int_{\mathbb{R}^d} \chi(\cdot, t) \varphi(\cdot, t) \mathrm{d} x-\int_{\mathbb{R}^d} \chi_0 \varphi(\cdot, 0) \mathrm{d} x=\int_0^t \int_{\mathbb{R}^d} \chi\left(\partial_t \varphi+(\f v \cdot \nabla) \varphi\right) \mathrm{d} x \mathrm{~d} t\label{eq:transport}
\end{equation}  for almost every $t \in\left[0, T \right]$ and all $\varphi \in C _{\mathrm{c}}^{\infty}\left(\mathbb{R}^d \times\left[0, T \right)\right)$.
\item The energy dissipation inequality
\begin{equation}\label{endisNavTwo}
\E^V (\chi,\f v , V)(t)
 +\int_0^t \int_{\mathbb{R}^d} {\mu(\chi)}\left|(\nabla\f v )_{\sym}\right|^2 \mathrm{~d} x \mathrm{~d} t 
\leq  \E (\chi,\f v , V)(0)
\end{equation}
is satisfied for almost every $ t \in [0, T ]$ with  the Varifold energy
\begin{equation}\label{enmeas}
\E^V (\chi,\f v , V)(t):=\int_{\mathbb{R}^d} \frac{1}{2} \rho(\chi(\cdot, t))|\f v(\cdot, t)|^2 \mathrm{~d} x+\sigma\left|V_t\right|\left(\mathbb{R}^d \times \mathbb{S}^{d-1}\right)
\end{equation}
is a nonincreasing function of time.
\item 
 The phase boundary $\partial\{\chi(\cdot, t)=0\}$ and the varifold $V$ satisfy the compatibility condition
\begin{align}
\int_{\mathbb{R}^d \times \mathbb{S}^{d-1}} \psi(x) s \mathrm{~d} V_t(x, s)=\int_{\mathbb{R}^d} \psi(x) \mathrm{d} \nabla \chi(x)\label{compatible}
\end{align}
for almost every $t \in\left[0, T \right)$ and every smooth function $\psi \in C _{\mathrm{c}}^{\infty}\left(\mathbb{R}^d\right)$.
\end{itemize}

\end{definition}
\begin{remark}
The existence of Varifold solutions for the two-phase Navier--Stokes equation was shown in a slightly different setting in~\cite{AbelsVari}. The weak-strong uniqueness of these solutions was proven in~\cite{HenselFischer}. Varifold solutions are a recurrent solution framework for geometrical flows such as the mean-curvature flow~\cite{TimHensel} or Mullin Sekerka Flow~\cite{MullinSekerka}. The proposed energy-variational framework has the potential to also simplify these solvability concepts by using the energy-variational structure of the system. 
\end{remark}

\begin{definition}[Energy-variational solution]\label{def:NavTwo}
A triple $(\chi,\f v , E)$  with
$$
\begin{aligned}
\f v & \in L^2(0, T ; \Hsig (\mathbb{R}^d ) )\cap L^{\infty}(0, T  ; \Ha(\mathbb{R}^d)), \\
\chi & \in L^{\infty}(0, T  ; \mathrm{BV}(\mathbb{R}^d ;\{0,1\})), \\
E & \in \BV
\end{aligned}
$$
is called an energy-variational solution to the free boundary problem for the Navier--Stokes equation for two fluids with initial data $\left(\chi_0,\f v _0\right)$ if 
\begin{itemize}

\item The auxiliary variable $E$ is an upper bound for the energy
\begin{equation}
E (t) \geq \E(\f v(t) , \chi(t)) \text{ for a.e. }t\in(0,T)  \text{ with } \E(\f v , \chi) : = \int
_{\R^d}  \frac{\rho(\chi)}{2} | \f v |^2 \de x + \sigma | \nabla \chi |(\R^d) \,.
\end{equation}
\item The energy-variational inequality 
\begin{multline}\label{envar:NavTwo}
\left [ E - \int_{\R^d} \rho(\chi)\f v \cdot \f \eta \de x - \int_{\R^d} \chi \varphi \de x \right ] \Bigg|_s^t + \int_s^t \int_{\R^d} \chi ( \t \varphi + ( \f v \cdot \nabla) \varphi) \de x \de t 
\\+ \int_s^t \int_{\R^d} \rho(\chi) \f v \cdot \t \f \eta + \rho(\chi) \f v \otimes \f v : \nabla \f \eta  +\mu(\chi) (\nabla \f v )_{\sym} : ( \nabla \f v- \nabla \f \eta) \de x   \de \tau 
\\
- \sigma \int_s^t\int_{\R^d} \left ( \text{Id}- \frac{\nabla \chi}{|\nabla \chi|}\otimes \frac{\nabla \chi}{|\nabla \chi|} \right ) : \nabla \f \eta \de | \nabla \chi|\de \tau 
\\+ \int_s^t \|( \nabla \f \eta)_{\sym,-} \|_{L^\infty(\R^d;\R^{d\times d})} \left  [ \E( \f v ,\chi) - E \right ]
\de \tau \leq 0 
\end{multline}
holds for all $ \varphi \in \C^1([0,T]; \C^\infty(\R^d))$, $\f \eta \in \C^1([0,T]; \C^\infty_{c,\sigma} (\R^d; \R^d))$ and a.e.~$0\leq s<t\leq T $ including $s=0$ with $ \chi(0)= \chi_0 $, $ \f v (0) = \f v_0$ and $E(0)= \E(\f v_0, \chi_0)$. 
\end{itemize}

\end{definition} 
\begin{theorem}
Let $ (\f v , \chi, E) $ be an energy-variational solution according to Definition~\ref{def:NavTwo}. Then there exists a varifold measure $ V \in L^\infty_{w^*}(0,T;\M(\R^d \times \Se^{d-1})) $ such that $(\f v , \chi , V)$ is a varifold solution in the sense of Definition~\ref{def:vari}. 
 
Let $(\f v , \chi , V)$ be a varifold solution in the sense of Definition~\ref{def:vari} such that there exists an $E\in \BV$ with $ E(0) = \E( \f v _0 , \chi_0)$ such that $E\geq \E(\chi,\f v , V)$ a.e. on $(0,T)$ and 
\begin{equation}\label{eq:enintwo}
 E\Big|_s^t + \int_s^t\int_{\R^d} \mu(\chi) | ( \nabla \f v )_{\sym}|^2 \de x \de \tau \leq 0 \quad \text{for all }T\geq t> s\geq 0\,.
\end{equation}
Then $(\f v , \chi , E)$ is an energy-variational solution in the sense of Definition~\ref{def:NavTwo}.
\end{theorem}
\begin{remark}
The assertion of the above theorem shows that the the equivalence of both solution concept depends on the  choice of the formulation of the energy inequality. 
The existence of energy-variational solution has not been shown so far. Usually the existence of energy-variational solutions can be shown by the same approximation strategy as the existence of varifold solutions~\cite{AbelsVari}, but using the energy-variational structure for the limit passage instead of varifolds. In~\cite{visco}, we designed a general minimizing movements scheme to show existence of energy-variational solutions to a large class of systems. This result is not applicable to the setting of the two-phase Navier--Stokes equations, but we plan to extend the results of~\cite{visco} to the framework of interface evolution in a future work. 

\end{remark}
\begin{proof}
We consider the energy-variational inequality~\eqref{envar:NavTwo} multiplied by $\alpha$ with $ \f \eta \equiv 0 $, $\varphi = \frac{\psi}{\alpha}$ for $\alpha >0 $ and $ \psi \in \C^1([0,T]; \C^\infty_c(\R^d))$. Taking the limit $\alpha \searrow 0$ 
we infer~\eqref{eq:transport}. 
Now, we consider~\eqref{envar:NavTwo} multiplied by $\alpha$ with $ \f \eta =  \frac{\f \zeta}{\alpha} $ for $\alpha >0 $, $\varphi \equiv 0 $  and $ \f \eta \in \C^1([0,T]; \C^\infty_c(\R^d;\R^d))$. Taking again the limit $\alpha \searrow 0$, we infer 
\begin{multline*}
 - \int_{\R^d} \rho(\chi)\f v \cdot \f \zeta \de x  \Big|_s^t  + \int_s^t \int_{\R^d} \rho(\chi) \f v \cdot \t \f \zeta + \rho(\chi) \f v \otimes \f v : \nabla \f \zeta  - \mu(\chi) (\nabla \f v )_{\sym} : \nabla \f \zeta \de x   \de \tau 
\\
- \sigma \int_s^t\int_{\R^d} \left ( \text{Id}- \frac{\nabla \chi}{|\nabla \chi|}\otimes \frac{\nabla \chi}{|\nabla \chi|} \right ) : \nabla \f \zeta \de | \nabla \chi|
+ \|( \nabla \f \zeta)_{\sym,+} \|_{L^\infty(\R^d;\R^{d\times d})} \left  [ \E( \f v ,\chi) - E \right ]
\de \tau \leq 0 
\end{multline*}
for all $\f \zeta \in \C^1([0,T]; \C^\infty_{c,\sigma} (\R^d))$ and a.e.~$0\leq s<t\leq T $ including $s=0$ with $ \chi(0)= \chi_0 $, $ \f v (0) = \f v_0$.
Considering the linear form $\f l : \mathcal V \to \R$ with $\mathcal{V}:=\{ \nabla \f \zeta | \f \zeta \in  \C^1([0,T];\C^1_{c,\sigma} (\R^d;\R^d))\}$ given by 
\begin{align*}
\langle \f l ,\nabla \f \zeta \rangle :=  &\int_{\R^d} \rho(\chi_0) \f v_0 \f \zeta(0) \de x + \int_0^T \int_{\R^d} \rho(\chi) \f v \cdot \t \f \zeta + \rho(\chi) \f v \otimes \f v : \nabla \f \zeta  - \mu(\chi) (\nabla \f v )_{\sym} : \nabla \f \zeta \de x   \de t 
\\
&- \sigma \int_0^T\int_{\R^d} \left ( \text{Id}- \frac{\nabla \chi}{|\nabla \chi|}\otimes \frac{\nabla \chi}{|\nabla \chi|} \right ) : \nabla \f \zeta \de | \nabla \chi| \de t 
\end{align*}
and the Minkowski functional 
$p: L^1(0,T;\C(\R^d; \R^{d\times d} ))\to [0,\infty)$ via 
\begin{equation}
\label{defp}
p(A ) := \int_0^T \|(A )_{\sym,+} \|_{L^\infty(\R^d;\R^{d\times d})} \left  [ E- \E( \f v ,\chi)  \right ]\de t \,.
\end{equation}
Note that the function $\f \zeta $ is uniquely determined by its gradient, since it is compactly supported. 
From~\eqref{envar:NavTwo}, we infer 
$ \langle \f l ,\nabla \f \zeta \rangle \leq p(\nabla \f \zeta)$ for all $\f \zeta \in \C^1([0,T];\C^1_{c,\sigma} (\R^d;\R^d))$. 

The Hahn--Banach theorem and the Riesz representation theorem implies that there exists $ \mathfrak{R}\in L^\infty_{w^*}(0,T;\M(\R^d;\R^{d\times d}))$ such that 
\begin{align}\label{definmeas}
\langle \f l ,\f \zeta \rangle  = - \int_0^T \int_{\R^d} \nabla \f \zeta : \de \mathfrak{R}\de t \quad\text{ and } \quad -\int_0^T \int_{\R^d} A : \de \mathfrak{R}\de t \leq p(A)
\end{align}
 for all $A \in  L^1(0,T;\C(\R^d; \R^{d\times d} ))$.  
Choosing $ A = B$ with $B = (B)_{\skw}\in \R^{d\times d}$ in~\eqref{definmeas} implies that 
 $0= \int_{\R^d} B : \de \mathfrak{R}$ for all $B\in \C(\R^d; \R^{d\times d}_{\skw})$, which is nothing else than $ \mathfrak{R}\in L^\infty_{w^*}(0,T;\M(\R^d;\R^{d\times d}_{\sym}))$. 
 Moreover, choosing $ A \in \C(\R^d; R^{d\times d}_{\sym,+}$ in~\eqref{definmeas} implies $0\geq  \int_{\R^d} B : \de \mathfrak{R}$ for all $B\in \C(\R^d; \R^{d\times d}_{\sym,-})$, which is nothing else than $ \mathfrak{R}\in L^\infty_{w^*}(0,T;\M(\R^d;\R^{d\times d}_{\sym,+}))$. 
Using the disintegration of measures, we can write $ {\mathfrak{R}} = B_{x,t} \mu_t$ with $ \mu \in L^\infty_{w^*}(0,T; \mathcal{M}(\R^d))$ and $ B_t \in L^1(\R^{d}, \mu ; \R^{d\times d})$ for~a.e. $t\in(0,T)$, where we can choose $ \mu$ and $B$ such that $ \tr (B(x,t))=1 $. 
 
For $\mu$-a.e.~$(x,t)$, we observe that  $ B(x,t) $ is symmetric positive definite such that there exists a decomposition into eigenvalues $\lambda_i (x,t)\in [0,1]$ and orthonormal eigenvectors $\f e_i(x,t)$ for $i \in \{1,\ldots,d\}$  such that 
 $ B(x,t) = \sum_{i=1}^d \lambda_i(x,t) e_i(x,t)\otimes e_i(x,t) $.
 Now, we can define the measure 
 \begin{align*}
 \nu_{(x,t)} \in \mathcal{M}(\Se^{d-1}) \text{ via } \nu_{(x,t)}(s ) =\sum_{i=1}^d \frac{\lambda_i(x,t)}{2} \left( \delta_{e_i(x,t)}(s)+\delta_{-e_i(x,t)}(s) \right). 
 \end{align*}
and set $$ V \in L^\infty_{w^*}(0,T;\mathcal{M}(\R^d\times \Se^{d-1}) )\text{ as } V_t(s,x) = \delta_{\frac{\nabla \chi_{x,t} }{|\nabla\chi_{(x,t)}|} x,t} (s)  |\nabla \chi |(x,t)+ 
\nu_{x,t} (s) \mu (x,t)$$ in the sense of measures. 
We note that there would be many other  possible choices for the measure $\nu$. It is just required that $ \int_{\Se^{d-1}} s \de  \nu_{(x,t)} (s) =0 $ and $ \int_{\Se^{d-1}} s\otimes s  \de  \nu_{(x,t)} (s) = B(x,t)$. But there are infinitely many measures for a given expectation and variance~(\textit{cf}.~Lemma~\ref{lem:Young}).  

Since $ \nu_{(x,t)}$ is symmetric its expectation vanishes, which implies now that~\eqref{compatible} holds. 
Additionally, we infer from~\eqref{definmeas} that 
\begin{align*}
\int_{\R^d} \rho(\chi_0) \f v_0 \f \zeta(0) \de x +\int_0^T& \int_{\R^d} \rho(\chi) \f v \cdot \t \f \zeta + \rho(\chi) \f v \otimes \f v : \nabla \f \zeta  - \mu(\chi) (\nabla \f v )_{\sym} : \nabla \f \zeta \de x   \de t 
\\
={}& \sigma \int_0^T\int_{\R^d} \left ( I - \frac{\nabla \chi}{|\nabla \chi|}\otimes \frac{\nabla \chi}{|\nabla \chi|} \right ) : \nabla \f \zeta \de | \nabla \chi|  - \int_{\R^d} \nabla \f \zeta :\de \mathfrak{R}_t(x)  \de t 
\\
={}& - \sigma \int_0^T\int_{\R^d\times \Se^{d-1}} s \otimes s :\nabla \f \zeta \de V(s,x) \de t 
\\
={}&  \sigma \int_0^T\int_{\R^d\times \Se^{d-1}}\left ( I -  s \otimes s \right ):\nabla \f \zeta \de V(s,x) \de t \,,
\end{align*}
where we used  $ I : \nabla \f \zeta = \di \f \zeta =0$ two times. 
This equation is exactly~\eqref{TwoPhaseNav}. 

By choosing $A = \phi I $ in~\eqref{definmeas}$_2$ with $\phi \in \C([0,T])$ with $\phi \geq 0$ implies by the definition of~$p$ in~\eqref{defp} that 
\[ \int_0^T \phi \int_{\R^d} \tr [ \de \mathfrak{R}] \de t \leq \int_0^T\phi [ E - \E( \f v , \chi)]\de t\quad \text{ for all }\phi \in \C([0,T]) \text{ with }\phi \geq 0 \,. \]
We note that we eqip the space of matrices $ \R^{d\times d}_{\sym}$ with the spectral norm $|A|_2 := \max_{i\in \{1,\ldots,d\}} |\lambda_i|$, where $\lambda _i$ denotes the $i$-th eigenvalue of $A\in \R^{d\times d}_{\sym}$. The trace-norm is than the dual norm of the spectral norm in the Frobenius product, \textit{i.e.}, $ \tr(A) = \sum _{i=1}^d \lambda_i= \sum _{i=1}^d |\lambda_i| = \sup_{B \neq 0,\, |B|_2\leq 1 } A:B $ for all $A\in \R^{d\times d}_{\sym,+}$.  
From the choice of $B$, we infer that $ \tr[\mathfrak{R}] = \mu $. 
By the definition of $V$, we find that $|V_t| (\R^d\times \Se^d) = |\nabla \chi_t|(\R^d)  + \mu_t(\R^d)$, which implies that 
\begin{align*}
E \geq \E( \f v, \chi) + \int_{\R^d} \tr(\mathfrak{R}) &= \int_{\Omega} \frac{\rho(\chi)}{2} | \f v|^2 \de x + | \nabla \chi _t|(\R^d) + \mu_t(\R^d) \\& = \int_{\Omega} \frac{\rho(\chi)}{2} | \f v|^2 \de x + V( \R^d \times \Se^{d-1}) \\& =\E^V(\f v , \chi , V)    \,,
\end{align*}
where this energy was defined in~\eqref{enmeas}. 
Thus, Choosing $\f \eta\equiv 0$ and $ \varphi\equiv 0$ in~\eqref{envar:NavTwo2} with $s=0$ implies~\eqref{endisNavTwo} such that $(\f v, \chi, V)$ is indeed a varifold solution in the sense of Definition~\ref{def:vari}.

Let $(\f v , \chi , V)$ be a varifold solution in the sense of Definition~\ref{def:vari} such that there exists an $E\in \BV$ with $ E(0) = \E( \f v _0 , \chi_0)$ with $E\geq \E(\chi,\f v , V)$ a.e. on $(0,T)$ and~\eqref{eq:enintwo} is fulfilled. 
We subtract from~\eqref{eq:enintwo} the equations~\eqref{TwoPhaseNav} and~\eqref{eq:transport} and add the equations~\eqref{TwoPhaseNav} and~\eqref{eq:transport} with $t=s$. This leads to 
\begin{multline}\label{envar:NavTwo2}
\left [ E - \int_{\R^d} \rho(\chi)\f v \cdot \f \eta \de x - \int_{\R^d} \chi \varphi \de x \right ] \Big|_s^t + \int_s^t \int_{\R^d} \chi ( \t \varphi + ( \f v \cdot \nabla) \varphi) \de x \de t 
\\+ \int_s^t \int_{\R^d} \rho(\chi) \f v \cdot \t \f \eta + \rho(\chi) \f v \otimes \f v : \nabla \f \eta  +\mu(\chi) (\nabla \f v )_{\sym} : ( \nabla \f v- \nabla \f \eta) \de x   \de \tau 
\\
%
%
-\sigma  \int_s^t \int_{\mathbb{R}^d \times \mathbb{S}^{d-1}}(\operatorname{Id}-s \otimes s): \nabla \f \eta \mathrm{d} V_\tau (x, s)
\de \tau \leq 0 
\end{multline}
for all $ \varphi \in \C^1([0,T]; \C^\infty(\R^d))$, $\f \eta \in \C^1([0,T]; \C^\infty_{c,\sigma} (\R^d; \R^d ))$ and a.e.~$0\leq s<t\leq T $ including $s=0$ with $ \chi(0)= \chi_0 $, $ \f v (0) = \f v_0$ and $E(0)= \E(\f v_0, \chi_0)$.
We observe that 
\begin{align*}
&-\int_{\R^d\times \mathbb{S}^{d-1}}(\operatorname{Id}-s \otimes s): \nabla \f \eta \mathrm{d} V_\tau (x, s)\de x  \\&{}= \int_{\R^d} \left ( 
\frac{\nabla \chi}{|\nabla \chi|}\otimes \frac{\nabla \chi}{|\nabla \chi|} \right ) : \nabla \f \eta \de | \nabla \chi|(x)
\\&\quad +\int_{\R^d\times \mathbb{S}^{d-1}}(
s \otimes s): \nabla \f \eta \left ( \mathrm{d} V_\tau (x, s) - \de \left(\delta_{\frac{\nabla \chi}{|\nabla \chi|}}(s) |\nabla \chi|(x)\right)\right )
\\
&{}\geq \int_{\R^d} \left ( 
\frac{\nabla \chi}{|\nabla \chi|}\otimes \frac{\nabla \chi}{|\nabla \chi|} \right ) : \nabla \f \eta \de | \nabla \chi|
-\|( \nabla \f \eta)_{\sym,-} \|_{L^\infty(\R^d;\R^{d\times d})} \left ( | V_\tau|(\R^d\times \mathbb{S}^{d-1}) -   |\nabla \chi|(\R^d)\right )
\\
&{}\geq -\int_{\R^d} \left ( \text{Id}- \frac{\nabla \chi}{|\nabla \chi|}\otimes \frac{\nabla \chi}{|\nabla \chi|} \right ) : \nabla \f \eta \de | \nabla \chi|
+\|( \nabla \f \eta)_{\sym,-} \|_{L^\infty(\R^d;\R^{d\times d})} \left  [ \E( \f v ,\chi) - E \right ]\,.
\end{align*}
We used that $| V_\tau|(\R^d\times \mathbb{S}^{d-1}) -   |\nabla \chi|(\R^d) = \E_V(\f v , \chi , V) - \E( \f v , \chi) \leq E - \E( \f v , \chi) $.
Inserting this inequality into~\eqref{envar:NavTwo2}  implies~\eqref{envar:NavTwo}, which proves the assertion. 
\end{proof}

\section{Quasilinear wave equation with convex energy}\label{sec:2}

In this section we consider the quasi-linear wave equation proposed in~\cite{DemouliniWeakStrong}
\begin{equation}
\frac{\partial^2 \f y}{\partial t^2}=\nabla \cdot \f  W (\nabla \f y)
\label{eq:elasticity}
\end{equation}
where $\f y: \Td \times [0,T] \rightarrow \mathbb{R}^d$ and $\f W= D G $ is the gradient of a strictly convex function $G:\R^{d \times d} \rightarrow[0, \infty)$, about which we make the following hypotheses:
\begin{hypothesis}\label{hypo:elast}
We assume that $G \in \C^3(\R^{d \times d} ;[0, \infty))$ such that 
\begin{itemize}

\item
$\exists m<M \in \R$ with  $m|\f Z|^2 \leqq D^2 G({\f F})[\f Z, \f Z] \leqq M|\f Z|^2, \text{ for all } \f F,\f Z,  \in \R^{d \times d},$
\item
$G(\f F)=g_0(\f F)+\frac{1}{2}|\f F|^2$ where $\lim _{|\f F| \rightarrow \infty} \frac{g_0(\f F)}{1+|\f F|^2}=0$,
\item
$\lim _{|\f F| \rightarrow \infty} \frac{\left|D G(\f F)\right|}{1+|\f F|^2}=0$,
\item
 $\left|D^3 G(\f F)\right| \leqq M$, for some $M>0$.
 \item $ G(\f 0 ) = 0 $ and $ G( \f F) \geq 0 $ for all $ \f F \in \R^{d\times d}$.

\end{itemize}
\end{hypothesis}
We note that under these assumptions imply that the mapping 
\begin{equation}\label{ass:convex}
\f Z \mapsto D G (F) : A + \frac{M}{m} | A|_2 G(\f F) \quad \text{is convex for all }A\in \R^{d\times d}\,.
\end{equation}
This is assured by calculating the second derivative of~\eqref{ass:convex} and using Hypothesis~\ref{hypo:elast}  the first and fourth bullet.

If $\f y$ is interpreted as a displacement vector, this equation could be
regarded as a model for elastodynamics, but the assumption of convexity is known to be physically unrealistic~\cite{rindler}. 
Introducing the notation $\f v =\partial_t \f y $ and $\f F =\nabla \f  y $, a  solution to~\eqref{eq:elasticity} in first order form consists of a pair $(\f v,\f  F)$  which solve
\begin{equation}\label{eq:elastCons}
\t \f v   =\di ( D G (\f F))\,, \qquad
\t \f F   =\nabla \f v \,, \qquad \partial_t \eta(\f v , \f F)+\di \f q(\f v, \f F) =0\,.
\end{equation}
A smooth solution fulfilling the first two equations, will automatically satisfy the third equation with 
 $\eta(\f v,\f  F)=\frac{1}{2}|\f v|^2+G(\f F)$ and $\f q (\f v, \f F)=( D G( \f F))^T \f v $, and we prescribe the initial data $ \f v (0) = \f v_0$ and $\f F(0)=\f F_0$ in $\Td$.

 We define $$ \V := L^2(\Td; 
 \R^d ) \times L^2(\Td; \R^{d\times d})\text{ and } \E : \V \to \R \text{ via } \E(\f v , \f F ) := \int_{\Td} \frac{1}{2} | \f v |^2 + G(\f v , \f F)\de x \,.$$ 
\begin{definition}[measure-valued solution]\label{def:elastmeas}
 A measure-valued solution to~\eqref{eq:elasticity} with initial data $\left(\f v _0,\f  F_0\right) \in\V $ consists of a pair $(\f v,\f  F) \in L^{\infty}(0.T;\V)$ and a Young measure  $\nu =\nu_{x, t}$, which is a 
family  of probability measures, \textit{i.e.,}
$$
\{ \nu_{(x,t)} \} \subset \mathcal{P}( \R^d \times \R^{d\times d} )\,, \text{ a.e.~in } \Td \times (0,T)
$$
 such that 
 \begin{subequations}
 \begin{itemize}
\item the equations~\eqref{eq:elastCons} are fulfilled in a measure-valued sense  
  \begin{align}
 \int_{\Td} \f v _0 \cdot \f \psi  \mathrm{d} x+\int_0^T \int_{\Td}  \f v \cdot \t \f \psi  \mathrm{~d} x \mathrm{~d} t&{}=\int _0^T \int _{\Td} \left\langle \nu_{(x,t)}, D G(\f S) : \nabla \f \psi \right\rangle  \mathrm{d} x \mathrm{~d} t\label{eq:elastmeas1} \\
 \int _{\Td} \f F_0 :\Psi  \mathrm{d} x+\int_0^T \int_{\Td}  \f F : \partial_t \Psi  \mathrm{~d} x \mathrm{~d} t&{} =\int_0^T \int_{\Td}  \f v \cdot  \di  \Psi  \mathrm{d} x \mathrm{~d} t,\label{eq:elastmeas2}
\end{align}
for all test functions $\psi  \in C_c^1\left( \Td \times (0,T) ; \R^d \right)$ and $\Psi  \in C_c^1\left( \Td \times (0,T) ; \R^{d\times d} \right)$,
\item the compatibility condition 
\begin{equation}
\int_{\Td} \f v(t) \cdot \f \psi \de x = \int_{\Td} \langle \nu _{(x,t)}, \f s\cdot \f \psi \rangle \de x  \qquad  \int_{\Td} \f F(t) : \Psi  \de x = \int_{\Td} \langle \nu _{(x,t)}, \f S\cdot \Psi  \rangle \de x
\end{equation}
for all test functions $\psi  \in C_c^1\left( \Td \times (0,T) ; \R^d \right)$ and $\Psi  \in C_c^1\left( \Td \times (0,T) ; \R^{d\times d} \right)$ and a.e.~$t\in(0,T)$,
\item  there exists a non-negative Radon measure $\boldsymbol{\gamma} \in L^\infty_{w^*}(0,T; \mathcal{M}^{+}(\Td ))$ such that 
\begin{align}\label{eq:elastmeas3}
\int_0^T\t \theta \left(\int_{\Td}\left\langle\nu _{x, t}, \frac{1}{2}| \f s | ^2 +G( \f S) \right\rangle \mathrm{d} x +\boldsymbol{\gamma}_t(\mathrm{d} x)\right)\mathrm{~d} t+ \theta(0) \E\left(\f v_0,\f  F_0\right) \mathrm{d} x \geq 0,
\end{align}
for all non-negative functions $\theta(t) \in C_c^1([0, T))$.
 \end{itemize}
 \end{subequations}
\end{definition}
With the usual abuse of notation, we define the 
 dual pairings for the young measure $\nu$ to be defined as 
\[
\langle \nu_{(x,t)} , g(x, t, \f s , \f S) \rangle := \int_{\R^d\times \R^{d\times d}} g(x,t, \f s , \f S) \de \nu_{(x,t)}(\f s , \f S) \,
\]
for all $ g \in \C(\Td \times [0,T]\times \R^d\times \R^{d\times d}) $ fulfilling 
$\lim _{(|\f v |+|\f F| )\rightarrow \infty} \frac{|g(x,t,\f v, \f F)|}{1+|\f v|^2+ |\f F|^2}< \infty $
for a.e.~$ (x,t) \in \Td\times (0,T)$.  
The dual pairing of a Young measure with a function should be interpreted as the Young measure $\nu $ in the dual pairing with the continuous function $(\f s , \f S) \mapsto h(\f s , \f S)$
for
 $ h :\C(  \R^{d}\times \R^{d\times d} ;\R)$ .

\begin{definition}\label{def:elastenvar}
A triple $ (\f v , \f F , E)$ is called an energy-variational solution to~\eqref{eq:elasticity} with initial values $(\f v _0 , \f F_0)$, if 
\begin{subequations}
\begin{align}
\f v \in L^\infty (0,T; L^2( \Td  ;\R^d)) \,,\quad
\f F \in L^\infty (0,T; L^2(\Td;\R^{d\times d })\,, \quad E \in \BV
\end{align}
such that $ E \geq \E( \f v ,\f F)$ for a.e. $t\in [0,T]$ as well as 
\begin{multline}\label{envarin:elast}
\left  [ E- \int_{\Td} \f v \cdot \f \varphi + \f F: \Psi  \de x \right ] \Bigg|_s^t + \int_s^t \int_\Omega \f v \cdot \t \f \varphi + \f F : \t \Psi  \de x \de \tau 
\\
- \int_s^t \int_{\Td} \frac{\partial G}{\partial \f F }( \f F) : \nabla \varphi + \f v \cdot \di \Psi  \de x \de \tau  + \int_s^t \frac{M}{m} \| \nabla \f \varphi \|_{L^2(\Td;\R^{d\times d})} \left [ \E(\f v ,\f F)- E\right ] \de \tau \leq 0 
\end{multline}
for a.e.~$T\geq t > s \geq 0$ including $s=0$ with $ ( \f v (0), \f F(0)) =( \f v_0, \f F_0)$, $E(0)= \E(\f v_0, \f F_0)$, and all $( \f\varphi ,\Psi  ) \in \C^1(\Td \times [0,T] ; \R^d \times \R^{d\times d})$. 
\end{subequations}
\end{definition}
\begin{theorem}\label{thm:mainelast}
Let $(\f v , \f F, E)$ be an energy-variational solution according to Definition~\ref{def:elastenvar}. Then there exists a young measure $\nu $ such that $(\f v , \f F ,\nu)$ is a measure-valued solution according to Definition~\ref{def:elastmeas}. 

Let $ (\f v, \f F, \nu)$ be a measure-valued solution, where~\eqref{eq:elastmeas3} is replaced by: there exists $E:[0,T]\to \R$ non-increasing such that $E(0) = \E(\f v_0, \f F_0)$ and 
\begin{subequations}\label{eq:enmeaselasat}
\begin{equation}\label{eq:enmeaelast1}
E(t) \geq  \int_{\Omega} \left \langle \nu_{x,t} , \frac{1}{2}|\f s|^2 + G(\f F) \right  \rangle \de x \quad \text{for a.e. }t\in(0,T)\,.
\end{equation}
as well as 
\begin{align}\label{est:nuelast}
\left \| \langle \nu
, D G(\f S) \rangle  - D G (\f F
) \right \|_{L^2(\Td;\R^{d\times d})} \leq  \frac{M}{m} \left [E(t) - \E(\f v(t) , \f F(t))\right ]\,.
\end{align}
\end{subequations}
Then $(\f v, \f F, E) $ is an energy-variational solution according to Definition~\ref{def:elastenvar}. 
\end{theorem}

\begin{proposition}\label{prop:exelast}
Let  $G:\R^{d \times d} \rightarrow[0, \infty)$ be a strictly convex twice continuously differentiable function fulfilling Hypothesis~\ref{hypo:elast} and let~\eqref{ass:convex} be fulfilled. 
 Then there exists an energy-variational solution in the sense of Definition~\ref{def:elastenvar}. 
\end{proposition}
\begin{remark}
The existence of measure-valued solutions was claimed in~\cite{DemouliniWeakStrong} but under the hypothesis (1--4). There also the weak-strong uniqueness result was proven to hold.  

This existence and weak-strong uniqueness result for energy-variational solutions follows  from an abstract existence result in~\cite{visco}, which will be shown below. The abstract result from~\cite{unsere} could also be used, since the flat torus $\Td$ is considered.  
We note that the abstract  assumptions in Proposition~\ref{prop:exelast} are relaxed in comparison the existence result for measure-valued solutions. A weak-strong uniqueness could also be shown under the Hypothesis~\ref{hypo:elast} for sufficiently regular smooth solutions in~\cite[Cor.~3.7]{visco}. 

\end{remark}

\begin{proof}[Proof of Theorem~\ref{thm:mainelast}]
Let $ (\f v , \f F , E)$ be an energy-variational solution according to Definition~\ref{def:elastenvar}. Choosing $ \f \varphi \equiv 0$ and $ \Psi  = \frac{1}{\alpha} \f \psi$ for $\alpha > 0 $ and $ \f \psi \in \C^1(\Td\times [0,T]; \R^{d\times d })$ in~\eqref{envarin:elast} and multiplying the resulting inequality by $ \alpha \in\R$ implies in the limit of $\alpha \searrow 0$ that the weak formulation~\eqref{eq:elastmeas2} is fulfilled. 
On the other hand, choosing $ \f \varphi = \frac{1}{\alpha } \f \phi $ for $ \alpha>0 $ and $ \f \phi \in \C^1(\Td\times [0,T]; \R^d )$ and $ \Psi  \equiv 0$ and mulltiplying the resulting inequality by $ \alpha$, we infer in the limit~$\alpha \searrow 0$ the inequality
\begin{align*}
- \int_{\Td} \f v \cdot \f \phi  \de x \Big|_0^T + \int_0^T \int_\Omega \f v \cdot \t \f \phi  - \frac{\partial G}{\partial \f F }( \f F) : \nabla \phi\de x + \frac{M}{m} \| \nabla \f \phi \|_{L^2(\Td;\R^{d\times d})} \left [ \E(\f v ,\f F)- E\right ] \de t \leq 0 \,
\end{align*}
for all 
$ \f \phi \in \C^1(\Td\times [0,T]; \R^d )$ and a.e.~$s<t\in (0,T)$. 
Defining $\mathcal{V}:= \mathcal{U}$ with 
\begin{align*}
\f l: \mathcal{U} \to \R \qquad \langle \f l , \f \varphi \rangle := - \int_{\Td} \f v \cdot \f \phi  \de x \Big|_0^T + \int_0^T \int_\Omega \f v \cdot \t \f \phi  - \frac{\partial G}{\partial \f F }( \f F) : \nabla \phi\de x\de t  
\end{align*}
as well as 
\begin{align*}
\mathfrak{p}: L^1(0,T; L^2(\Omega ; \R^{d\times d}) \qquad \mathfrak{p}(\f A):= \frac{M}{m} \| \f A \|_{L^2(\Td;\R^{d\times d})} \left [ E- \E(\f v ,\f F)\right ]\,,
\end{align*}
we may apply Lemma~\ref{lem:hahn} in order to infer that there exists 
$\mathfrak R \in L^\infty (0,T;L^2(\Td ; \R^{d\times d}
)) $
such that $\langle - \mathfrak{R}, \f A \rangle \leq \mathfrak{p}(\f A ) $ and $\langle \f l , \f \varphi \rangle = \langle - \mathfrak{R}, \f \varphi \rangle $. 
With the definition of $\f l$ and we find that
\begin{align}\label{linineq}
- \int_{\Td} \f v \cdot \f \phi  \de x \Big|_0^T + \int_0^T \int_\Omega \f v \cdot \t \f \phi  - \frac{\partial G}{\partial \f F }( \f F) : \nabla \phi\de x\de t  + \int_0^T \int_{\Omega} \nabla \f \phi :  \mathfrak{R}\de x  \de t = 0 \,.
\end{align}
Additionally, taking the supremum in $\f A$  we observe that  \[ \| \mathfrak{R}(t) \| _{L^2(\Td; \R^{d\times d} )} \leq \frac{M}{m} \left [E(t) - \E(\f v(t) , \f F(t))\right ]\,.\]

Now, we want to construct the Young measure $\nu$ via Lemma~\ref{lem:Young}.
%
Therefore, we define 
 $g:\R^ d \times  \R^{d\times d} \to \R^{d\times d}$ via $ g ( \f s , \f S  ) := \frac{\partial G}{\partial \f F}(\f S)$ as well as $\eta (\f s , \f S) = \frac{1}{2}| \f s |^2 + G( \f S)$ and observe that $\eta$ fulfills all requirements of Lemma~\ref{lem:Young} thanks to Hypothesis~\ref{hypo:elast}. 
 The growth of $\eta$ allows to identify with $ \Omega = \Td$ that $ V = L^2(\Td; \R^d\times \R^{d\times d})$. We define $ 
  W : = L^2(\Td;\R^{d\times d})$ and   equip it with the norm with $\| \f S\|_{W}  =\frac{M}{m}\|\f S\|_{L^2(\Omega;\R^{d\times d})}  $. 
 Since $G: W \to [0,\infty)$ is a convex lower semicontinuous function that is coercive, its subdifferential, which coincides with the standard differential, which is $ g : W \to W^*=W$ is surjectiv. 
Considering Hypothesis~\ref{hypo:elast} it is a standard matter to verify that the bound~\eqref{bound} is fulfilled.
 
 We apply Lemma~\ref{lem:Young} for $ \f y_0 = (\f v , \f F)$, $ \f z _0 = \mathfrak{R}$ and $ x_0 = E-\E(\f v, F)$. 
  According to Lemma~\ref{lem:Young} there exists a $ \nu \in L^\infty_{w^*}(0,T;L^1(\Omega)\times \mathcal{M}(\R^d\times \R^{d\times d}))$ and a $\gamma \in L^\infty_{w^*}(0,T;\mathcal{M}^+(\Omega))$  such that 
\begin{align*}
1 =  \langle  \nu_{x,t} , 1\rangle  \,, \quad 
(\f v(x,t), \f F(x,t)) =\langle  \nu_{x,t} , (\f s , \f S)\rangle  \,, \quad \text{and } \\\quad  \frac{\partial G}{\partial \f F}(\f F(x,t))  - \mathfrak{R}(x,t)= \langle \nu_{x,t},  \frac{\partial G}{\partial \f F}\rangle  
\end{align*} 
a.e.~in $\Omega \times (0,T)$ as well as 
\begin{align*}
\E(\f v , \f F) + \frac{m}{M}\| \mathfrak{R}\|_{L^2(\Omega;\R^{d\times d})}  + E-\E(\f v, F) = \int_{\Td} \left  \langle \nu , \frac{1}{2}|\f s |^2 + G( \f S) \right \rangle  \de x + \int_{\Td} \de \gamma_t(x) \,.
\end{align*}
This inserted into~\eqref{linineq} implies the formulation~\eqref{eq:elastmeas1}. The monotonicity of $E$, stemming from choosing $\f \varphi \equiv \f 0$ as well as $\Psi  \equiv \f 0$ in~\eqref{envarin:elast},  then implies~\eqref{eq:elastmeas3}.

This proves the first assertion of  Theorem~\ref{thm:mainelast}.
We note that due to the estimate on~$\mathfrak{R}$, and its identification, we find that the estimate~\eqref{est:nuelast} holds.
Thus, $ (\f v , \f F , \nu)$ is a measure-valued solution according to Definition~\ref{def:elastmeas}. 

Let $ (\f v, \f F, \nu)$ be a measure-valued solution, where~\eqref{eq:elastmeas3} is replaced by: there exists $E:[0,T]\to \R$ non-increasing such that $E(0) = \E(\f v_0, \f F_0)$ and~\eqref{eq:enmeaselasat} is fulfilled.

Due to the convexity of $\eta$, we infer from Jensen's inequality and~\eqref{eq:enmeaelast1}  that
\begin{equation}
E(t) \geq \int_{\Td} \langle \nu_{x,t}, \eta \rangle \de x \geq \int_\Omega \eta \left ( \langle \nu_{x,t}, \f A \rangle \right ) \de x = \E(\f v (t), \f F(t)) \quad \text{for a.e. } t\in (0,T)\,.
\end{equation}
Now, we subtract from $ E(t) - E(s) \leq 0$ the two equations~\eqref{eq:elastmeas1} and~\eqref{eq:elastmeas2} and add as well as subtract the term $ \frac{\partial G}{\partial \f F}(\f F)$
 such that 
\begin{multline*}
\left  [ E- \int_{\Td} \f v \cdot \f \varphi + \f F: \Psi  \de x \right ] \Big|_s^t + \int_s^t \int_\Omega \f v \cdot \t \f \varphi + \f F : \t \Psi  \de x \de \tau 
\\
- \int_s^t \int_{\Td}\frac{\partial G}{\partial \f F}(\f F) : \nabla \f \varphi + \f v \cdot \di \Psi   - \left (\frac{\partial G}{\partial \f F}(\f F) - \left \langle \nu_{x,t}, \frac{\partial G}{\partial \f F }\right \rangle \right ) : \nabla\f  \varphi \de x \de \tau 
%
  = 0 \,. 
\end{multline*}
The estimate~\eqref{est:nuelast2} now implies the energy-variational inequality~\eqref{envarin:elast}. 
\end{proof}
\begin{proof}[Proof of Proposition~\ref{prop:exelast}]
We apply the abstract theorems from~\cite{visco} in order to prove the existence and weak-strong uniqueness of energy-variational solutions in the sense of Definition~\ref{def:elastenvar}. 

Therefore, we consider the Banach space $ \V $ given by $ \V = L^2(\Td; \R^d \times \R^{d\times d})$ and the 
energy  $ \E : \V \to \R $ via $ \E(\f v , \f F ) := \int_{\Td} \frac{1}{2} | \f v |^2 + G( \f F) - G ( \f F_{\min})\de x $, where $ \f F_{\min}$ is the unique minimizer of $G$. 
We note that $\V$ is a reflexive Banach space and $\E$ is a strictly convex superlinear energy on $\V$. 
We define $ \Y: =\C^1(\Omega; \R^d\times \R^{d\times d})$ and the operator $ A : \V \to \Y^*$ via 
\[
\langle A( \f v, \f F) , ( \f \varphi , \Psi) \rangle = \int_{\Td} \frac{\partial G}{\partial \f F }( \f F) : \nabla \varphi + \f v \cdot \di \Psi  \de x  \,.
\]
Moreover, we define the convex continuous regularity weight $ \mathcal{K}: \Y \to [0,\infty)$ via $ \mathcal{K}( \f \varphi  , \Psi) := \frac{M}{m} \| \nabla \f \varphi \|_{L^2(\Td;\R^{d\times d})} $. 
The convexity assumption~\eqref{ass:convex} implies that the the mapping
\begin{align*}
\V \ni ( \f v, \f F) \mapsto  - \langle A( \f v, \f F) , ( \f \varphi , \Psi ) \rangle  + \mathcal{K}( \f \varphi  , \Psi ) \E(\f v, \f F) 
\end{align*}
is convex for every $ ( \f \varphi  , \Psi ) \in \Y$, which is the main assumption in~\cite[Thm.~3.3]{visco}.
Moreover, we compute the subdifferential of the convex conjugate of the energy 
\[
 \partial \E^* ( \f \varphi  , \Psi) = \begin{pmatrix}
 \f \varphi  \\D G^* (\Psi )
 \end{pmatrix}\,,
\]
where $ G^*$ denotes the convex conjugate of $G$ as a function on $\R^{d\times d}$. 
Since $ G\in\C^2(\R^{d\times d})$ we infer that $D G^* (\Psi ) \in \C^1(\R^{d\times d }; \R^{d\times d })$. Via Fenchel's equivalence, we find 
$$
\langle A( D \E^*  ( \f \varphi  , \Psi ),  ( \f \varphi  ,\Psi )\rangle = \langle  A ( \f v , \f F) , D\E(\f v , \f F)\rangle  =\int_{\Td} \frac{\partial G}{\partial \f F }( \f F) : \nabla \f v  + \f v \cdot \di \frac{\partial G}{\partial \f F }( \f F) \de x = 0 
$$
for all $ ( \f \varphi  , \Psi ) \in \Y $, which implies $ (\f v , \f F) \in \Y$ such that the calculation becomes obvious with an integration-by-parts. 
Thus, all assumptions of Theorem~3.3 in~\cite{visco} are fulfilled such that there exists an energy-variational solution in the sense of Definition~\ref{def:elastenvar}.
\end{proof}
\section{Polyconvex elastodynamics}\label{sec:poly}
Similar to the previous section, we consider the system
\begin{equation}
\frac{\partial^2\f  y}{\partial t^2}=\nabla \cdot \f W(\nabla \f y)\label{polyconvexeq}
\end{equation}
with $\f y: \Td \times [0,T] \rightarrow \mathbb{R}^d$ being the displacement. In this section, we restrict ourselves to 3 dimensions, \textit{i.e.,} we set $d=3$. But $\f W : \R^{d\times d} \to \R^{d\times d}$ is not the gradient of a convex function anymore, but the Piola--Kirchoff stress tensor. 
We assume that the $\f S$ is the gradient of a function $\sigma : \R^{d\times d} \to \R$ 
where 
$$
\sigma(\f F)=G (\f F, \operatorname{cof} \f F, \operatorname{det} \f F),
$$
and $G  :\R^{d\times d }\times \R^{d \times d } \times \R \to \R$ is  a strictly convex function in all its variables, \textit{i.e.,} $\f F, \operatorname{cof} \f F$ and $\operatorname{det} \f F$ ($\operatorname{cof}\f  F$ is the matrix of the cofactors of $F$).

According to an observation of Qin~\cite{qin}, solutions to~\eqref{polyconvexeq} fulfill the  conservation laws
\begin{subequations}\label{conservpoly}
\begin{align}\label{conservpolyone}
{\partial_ t} \f v  &{}=\di \f  S(\f F) \quad &&\text{in } \Td\times [0,T]
\\
\partial _t\f  F &{} =\nabla \f v \quad &&\text{in } \Td\times [0,T]
\\
{\partial _t} \operatorname{det} \f F & =\di \left( (\operatorname{cof} \f F)^T\f  v\right)\quad &&\text{in } \Td\times [0,T]\,,  \\
{\partial_ t}(\operatorname{cof} \f F) & =\curl (\f v \times \f F)\quad &&\text{in } \Td\times [0,T]\,,
\\
( \f v(0), \f F(0))&{} = (\f v_0, \f F_0) \quad &&\text{in }\Td\,.
\end{align}
\end{subequations}
where we use the cross product for a vector matrix product, \textit{i.e.,}
$ (\f a \times \f A )_{ij}= \sum_{k,l=1}^d\f \epsilon_{ikl} \f a_k \f A_{lj}$ for $\f a\in\R^d$ and $\f A\in\R^{d\times d}$ as well as
$ (\nabla\times \f A )_{ij}= \sum_{k,l=1}^d\f \epsilon_{jkl} \partial_{x_k} \f A_{il}$ for $\f a\in\R^d$ and $\f A\in\C^1(\Omega;\R^{d\times d})$.
such that 
$ (\curl (\f v \times \f F))_{ij}=\sum_{klmn=1}^d \f \epsilon_{jkl} \partial_{x_k}  (\f \epsilon_{imn} \f v_m \f F_{nl})$. 
Where $\f\varepsilon$ is the Levi--Cevita symbol introduced in~\eqref{levi}.
These identities~\eqref{conservpoly} are derivable directly from the equation in~\eqref{polyconvexeq}. We may consider the system in terms of new variables $(\f v,\f  F,\f  Z, w)$ with $\f Z=\operatorname{cof}\f  F$ and $w=\operatorname{det} \f F$. 
The cofactor matrix $\operatorname{cof} \f F$ and the determinant $\operatorname{det}\f  F$ are given by
$$
\begin{aligned}
(\operatorname{cof}\f  F)_{i \alpha} & =\sum_{j,k,\beta,\gamma=1}^3\frac{1}{2} \f\varepsilon_{i j k} \f\varepsilon_{\alpha \beta \gamma} \f F_{j \beta} \f F_{k \gamma}, \\
\operatorname{det} \f F & =\frac{1}{6} \sum_{i,j,k,\alpha,\beta,\gamma=1}^3\f \varepsilon_{i j k} \f\varepsilon_{\alpha \beta \gamma} \f F_{i \alpha} \f F_{j \beta} \f F_{k \gamma}=\frac{1}{3}(\operatorname{cof}\f  F)_{i \alpha} \f F_{i \alpha} .
\end{aligned}
$$
An advantage of this formulation is that the enlarged system admits a strictly convex energy:
$$
\eta : \R^d \times \R^{d\times d} \times \R^{d\times d} \times \R \,; \quad \eta(\f v,\f  F,\f  Z, w)=\frac{1}{2}|\f v|^2+G(\f F, \f Z, w) .
$$
For convenience, we define 
$\zeta : \R^d \times \R^{d\times d} \times \R^{d\times d} \times \R^{d\times d}$  by 
$$
\begin{aligned}
\zeta _{i \alpha}\left(\f F,\f  Z, w \right) & =\frac{\partial G}{\partial\f  F_{i \alpha}}+\sum_{j,k,\beta,\gamma=1}^d\frac{\partial G}{\partial \f Z_{k \gamma}} \varepsilon_{i j k} \varepsilon_{\alpha \beta \gamma} \f F_{j \beta}+\left(\operatorname{cof} \f F\right)_{i \alpha} \frac{\partial G}{\partial w} 
\,,
\end{aligned}
$$
where we used 
 the formulas on derivatives of determinants and cofactor matrices,
$$
\begin{gathered}
\frac{\partial \operatorname{det} \f F}{\partial \f  F_{i \alpha}}=(\operatorname{cof} \f F)_{i \alpha}, \\
\frac{\partial(\operatorname{cof} \f F)_{i \alpha}}{\partial \f F_{j \beta}}=\sum_{k,\gamma=1}^d\varepsilon_{i j k} \varepsilon_{\alpha \beta \gamma} \f F_{k \gamma}\,.
\end{gathered}
$$

This leads to  the equations of elastodynamics in form of  the enlarged system
\begin{subequations}
\begin{align}\label{conservpolytwo}
{\partial_ t} \f v  &{}=\di \f  S(\f F) \quad &&\text{in } \Td\times [0,T]\,,
\\
\partial _t\f  F &{} =\nabla \f v \quad &&\text{in } \Td\times [0,T]\,,
\\
{\partial _t} w  & =\di \left( (\operatorname{cof} \f F)^T\f  v\right)\quad &&\text{in } \Td\times [0,T]\,,  \\
{\partial_ t}\f Z  & =\curl (\f v \times \f F)\quad &&\text{in } \Td\times [0,T]\,,
\\
( \f v(0), \f F(0), \f Z (0) , w(0))&{} 
= (\f v _0 , \f F_0 , \cof \f F_0 , \det \f F_0) \quad &&\text{in }\Td\,.
\end{align}
\end{subequations}

\begin{hypothesis}\label{hyp:elast}
We assume the  following convexity and growth assumptions on $G$, let $p \in(4, \infty)$ and $q, r \in$ $[2, \infty)$ be fixed 
\begin{itemize}
\item $G \in C^3\left(\R^{3 \times 3} \times \R^{3 \times 3} \times \mathbb{R} ;[0, \infty)\right)$ is a strictly convex function \textit{i.e.}, $\exists \gamma>0$ such that $D^2 G \geqq \gamma>0$,
\item $G(\f F,\f  Z, w) \geqq c_1|\f F|^p+c_2|\f Z|^q+c_3|w|^r-c_4$,
\item $G(\f F, \f Z, w) \leqq c\left(|\f F|^p+|\f Z|^q+|w|^r+1\right)$,
\item $\left|\partial_{\f F} G\right|^{\frac{p}{p-1}}+\left|\partial_{\f Z} G\right|^{\frac{p}{p-2}}+\left|\partial_w G\right|^{\frac{p}{p-3}} \leqq C\left(|\f F|^p+|\f Z|^q+|w|^r+1\right)$.
\end{itemize}
\end{hypothesis}

An example of a function satisfying Hypothesis~\ref{hyp:elast} is $\tilde{G}(\f F,\f  Z, w)=$ $\alpha|\f F|^6+|\f F|^2+\beta|\f Z|^3+|\f Z|^2+w^2$ for $\alpha, \beta$ non-negative.

\begin{definition}[Measure-valued solution]\label{def:polyelastmeas}
 A measure-valued solution to~\eqref{eq:elasticity} with initial data $\left(\f v_0,\nabla \f  y_0, \cof \nabla \f  y_0 , \det \nabla \f  y_0\right) \in$ $L^2(\Td;
 \R^d) \times L^p(\Td; \R^{d\times d})\times L^q(\Td; \R^{d\times d})\times L^r(\Td;\R)$ consists of a  tuple $
  (\f v,\f  F,\f  Z ,w ) \in L^{\infty}\left(0,T;L^2(\Td;
 \R^d) \times L^p(\Td; \R^{d\times d})\times L^q(\Td; \R^{d\times d})\times L^r(\Td;\R)\right) $ and a Young measure $\boldsymbol{v}=$ $\left(\boldsymbol{v}_{x, t}\right)_{x, t \in \Td\times (0,T)}$, which is a 
family  of probability measures, \textit{i.e.,}
$$
\{ \nu_{(x,t)} \} \subset \mathcal{P}( \R^d \times \R^{d\times d} \times \R^{d\times d} \times \R )\,, \text{ a.e.~in } \Omega \times (0,T)
$$
 such that  
 \begin{subequations}
  \begin{itemize}
\item the equations~\eqref{conservpolyone} are fulfilled in a measure-valued sense  
  \begin{align}
 \int_{\Td} \f v _0 \cdot \f \psi  \mathrm{d} x+\int_0^T \int_{\Td}  \f v \cdot \t \f \psi  \mathrm{~d} x \mathrm{~d} t&{}=\int _0^T \int _{\Td} \left\langle \nu
 , \zeta (\f s , \f S, \f R, r) : \nabla \f \psi \right\rangle  \mathrm{d} x \mathrm{~d} t\,,\label{eq:polyelastmeas1} \\
 \int _{\Td} \f F_0 : \Psi \mathrm{d} x+\int_0^T \int_{\Td}  \f F : \partial_t \Psi  \mathrm{~d} x \mathrm{~d} t&{} =\int_0^T \int_{\Td}  \f v \cdot  \di \Psi   \mathrm{d} x \mathrm{~d} t\,,\label{eq:polyelastmeas2}\\
 \int _{\Td} \f Z_0 : \Psi \mathrm{d} x+\int_0^T \int_{\Td}  \f Z : \partial_t \Psi  \mathrm{~d} x \mathrm{~d} t&{} =\int_0^T \int_{\Td}  (\f v \times \f F )  : \curl\Psi   \mathrm{d} x \mathrm{~d} t\,,\label{eq:polyelastmeas3}
 \\
 \int _{\Td} w_0 : \phi \mathrm{d} x+\int_0^T \int_{\Td}  w : \partial_t  \phi \mathrm{~d} x \mathrm{~d} t&{} =\int_0^T \int_{\Td} (\operatorname{cof} \f F)^T \f v \cdot \nabla \phi    \mathrm{d} x \mathrm{~d} t\,,\label{eq:polyelastmeas4}
\end{align}
for all  $\f \psi  \in C_c^1\left( \Td \times [0,T) ; \R^d \right)$, $\Psi   \in C_c^1\left( \Td \times [0,T) ; \R^{d\times d} \right)$,
and $\phi \in \C^1(\Td \times [0,T))$. 
\item the compatibility conditions
\begin{align}\notag
\int_{\Td} \f v(t) \cdot \f \psi \de x =& \int_{\Td} \langle \nu_t 
, \f s\cdot \f \psi \rangle \de x &\qquad  \int_{\Td} \f F(t) : \Psi  \de x =& \int_{\Td} \langle \nu_t 
, \f S: \Psi  \rangle \de x&
\\
\int_{\Td} \f Z(t) : \Psi  \de x =& \int_{\Td} \langle \nu_t 
, \f R: \Psi  \rangle \de x  &\qquad  \int_{\Td} w (t)  \phi \de x =& \int_{\Td} \langle \nu_t 
 , r \phi \rangle \de x&\label{eqpoly:cons}
\end{align}
hold for all test functions $\psi  \in C_c^1\left( \Td \times (0,T) ; \R^d \right)$, $\Psi  \in C_c^1\left( \Td \times (0,T) ; \R^{d\times d} \right)$,  $\phi \in \C^1(\Td \times [0,T))$, and a.e.~$t\in(0,T)$,
\item  it holds \begin{equation}\label{eqpoly:inden}
 ( \f F, \f Z , w) = ( \f F ,\operatorname{cof} \f F , \det \f F ) =  ( \langle \nu , \f S\rangle , \cof \langle \nu , \f S\rangle  , \det \langle \nu , \f S\rangle ) 
 \end{equation}
a.e.~in $\Td \times (0,T)$,
\item  there exists a non-negative Radon measure $\boldsymbol{\gamma} \in L^\infty_{w^*}(0,T; \mathcal{M}^{+}(\Td ))$ such that 
\begin{align}\label{eqpoly:elastmeas3}
\int_0^T\int_{\Td}\t \theta \left(\left\langle\nu , \eta\right\rangle \mathrm{d} x +\boldsymbol{\gamma}_t(\mathrm{d} x) \right)\de t +\int_{\Td} \theta(0) \eta\left(\f v_0,\nabla \f  y_0, \cof \nabla \f  y_0 , \det \nabla \f  y_0 \right) \mathrm{d} x \geq 0,
\end{align}
for all non-negative functions $\theta  \in C_c^1([0, T))$,
\item there exists $\f y\in L^\infty(0,T; W^{1,p}(\Td; \R^d) ) \cap W^{1,2}(0,T; L^2(\Td; \R^d))$ such that $ \f v = \t \f y$ and $ \f F = \nabla \f y$.
 \end{itemize}
  \end{subequations}

\end{definition}
The dual pairings for the Young measure $\nu$ is defined as 
\[
\langle \nu_{(x,t)} , g(x, t, \f s , \f S, \f R , r ) \rangle := \int_{\R^d\times \R^{d\times d}\times \R^{d\times d} \times \R } g(x,t, \f s , \f S, \f R , r) \de \nu_{(x,t)}(\f s , \f S, \f R , r) \,
\]
for all $ g \in \C(\Td \times [0,T]\times \R^d\times \R^{d\times d}\times \R^{d\times d}\times \R ) $ fulfilling 
$\lim _{| \f v|+|\f F| \rightarrow \infty} \frac{g(x,t,\f v, \f F, \f Z , w )}{1+|\f v|^2+ G(\f F, \f Z , w) }=0$
for a.e.~$ (x,t) \in \Td\times (0,T)$.
We define the energy $$ \E:  L^2({\Td}; \R^d)\cap  W^{1,p}(\Td;\R^{d } )\to \R \text{ via }\E(\f v , \f y) = \int_{\Td} \eta ( \f v , \nabla \f y , \cof \nabla \f y , \det \nabla \f y)\de x \,.$$

\begin{definition}[Energy-variational solution]\label{def:polyelastenvar}
A tuple  $ (\f y ,  E)$ is called an energy-variational solution to~\eqref{eq:elasticity}, with initial data $$\left(\f v_0,\nabla \f  y_0, \cof \nabla \f  y_0 , \det \nabla \f  y_0\right) \in L^2(\Td;
 \R^d) \times L^p(\Td; \R^{d\times d})\times L^q(\Td; \R^{d\times d})\times L^r(\Td;\R)$$ if 
\begin{subequations}
\begin{align}
\f y \in W^{1,\infty} (0,T; L^2({\Td}; \R^d))\cap L^\infty(0,T; L^p(\Td;\R^{d\times d })
\,, \quad E \in \BV
\end{align}
with $ \operatorname{cof} \nabla \f y  \in  L^\infty (0,T;L^q(\Td;\R^{d\times d }) )$ as well as $ \det \nabla \f y  \in L^\infty(0,T; L^r(\R))$
such that $ E \geq \E( \f v ,\f F)$ for a.e. $t\in [0,T]$ as well as 
\begin{multline}\label{envarin:polyelast}
\left  [ E- \int_{\Td} \t \f y  \cdot \f \varphi + \nabla \f y  : \Psi  + \operatorname{cof} \nabla \f y  :  \Xi+ \det \nabla \f y  \phi  \de x \right ] \Big|_s^t 
\\
+ \int_s^t \int_{\Td} \t \f y \cdot \t \f \varphi + \nabla \f y   : \t \Psi + \operatorname{cof} \nabla \f y  :\t   \Xi+ \det \nabla \f y   \t \phi \de x \de \tau 
\\
- \int_s^t \int_{\Td} \zeta ( \nabla \f y ,  \operatorname{cof} \nabla \f y , \det \nabla \f y  ) : \nabla \f \varphi +  \t \f y  \cdot \di \Psi  
 \de x \de \tau \\
- \int_s^t \int_{\Td} 
(\t \f y  \times \nabla \f y  )  : \curl \Xi  
+(\operatorname{cof} \nabla \f y )^T \t \f y  \cdot \nabla \phi  
%
\de x \de \tau   \\+ \int_s^t c\| \nabla \f \varphi \|_{L^p(\Td;\R^{d\times d})} \left [ \E(\t \f y  , \f y )- E\right ] \de \tau \leq 0 
\end{multline}
for a.e.~$T\geq t > s \geq 0$ including $s=0$ with $ ( \t \f y (0),  \f y(0)) =( \f v_0, \f y_0)$ in $ L^2(\Td;\R^d) \times W^{1,p}(\Td; \R^{ d}) $ and all $( \f\varphi ,\Psi ,  \Xi,\phi  ) \in \C^1(\Td \times [0,T] ; \R^d \times \R^{d\times d}\times \R^{d\times d} \times \R)$. 
\end{subequations}
\end{definition}


\begin{theorem}\label{thm:mainelast2}
Let Hypothesis~\ref{hyp:elast} be fulfilled.
Let $(\f y,  E)$ be an energy-variational solution according to Definition~\ref{def:polyelastenvar}.

 Then there exists a Young measure $\nu $ such that $(\t \f v , \nabla \f y  ,\operatorname{cof} \nabla \f y , \det \nabla \f y ,\nu)$ is a measure-valued solution according to Definition~\ref{def:polyelastmeas}. 

Let $ (\f v, \f F,\f Z,w ,  \nu)$ be a measure-valued solution, where~\eqref{eq:polyelastmeas3} is replaced by: there exists $E:[0,T]\to \R$ non-increasing such that $E(0) = \E(\f v_0, \nabla \f y_0)$ and 
\begin{equation}
E(t) \geq \int_{{\Td}} \langle \nu_{x,t} , \eta \rangle \de x \quad \text{for a.e. }t\in(0,T)\,.\label{strongeninelast}
\end{equation}
as well as 
\begin{align}\label{est:nuelast2}
\left \| \langle \nu , \zeta( \f S, \f R, r) \rangle - \zeta( \f F, \cof \f F, \det \f F ) \right \|_{\mathbb{W}^*} \leq  c \left [E(t) - \E(\f v(t) , \f F(t)
 )\right ]\,
\end{align}
for a.e.~$t\in(0,T)$, where  $\mathbb 
  W : = \{  \f A \in L^p(\Td; \R^{d\times d}) \mid \exists \f y \in W^{1,p}( \Td; \R^{d}) \text{ such that } \nabla \f y = \f A  \}$.

Then $(\f v, \f F, E) $ is an energy-variational solution according to Definition~\ref{def:polyelastenvar}. 
\end{theorem}
\begin{remark}
The existence of measure-valued solutions in the sense of Definition~\ref{def:polyelastmeas} was proven in~\cite{DemouliniExist} under the Hypothesis~\ref{hyp:elast}, the weak-strong uniqueenss was shown in~\cite{DemouliniWeakStrong}. 

The energy-variational formulation in the sense of Definition~\ref{def:polyelastenvar} has the advantage, that no Young-measures are needed to describe the solutions, but moreover the solution does not need the additional variables $ \f Z $ and $w$. These are in the Definition~\ref{def:polyelastmeas} only needed within the Young-measure-valued terms in~\eqref{eq:polyelastmeas1} and~\eqref{eq:elastmeas3} due to the identification~\eqref{eqpoly:inden}.

\end{remark}

\begin{proof}[Proof of Theorem~\ref{thm:mainelast2}]
Let $ (\f y , E)$ be an energy-variational solution according to Definition~\ref{def:polyelastenvar}.  We set $ \f v = \t \f y $ and $ \f F = \nabla \f y$. 
We observe that the formulation~\eqref{envarin:polyelast} is linear in the test functions $ \Psi  $, $\Xi $, as well as $\phi$. As beforehand in the previous proofs, we can deduce that the weak formulations~(\ref{eq:polyelastmeas2}--\ref{eq:polyelastmeas4}) are fulfilled with the identifications according to~\eqref{eqpoly:cons} and \eqref{eqpoly:inden}.

On the other hand, choosing $ \f \varphi = \frac{1}{\alpha } \f \phi $ for $ \alpha>0 $ and $ \f \phi \in \C^1(\Td\times [0,T]; \R^d )$, $ \Psi  \equiv 0$, $\Xi  \equiv 0$, as well as $\phi \equiv 0$ and multiplying the resulting inequality by $ \alpha$, we infer in the limit~$\alpha \searrow 0$ the inequality
\begin{align}\label{polylinineq}
- \int_{\Td} \f v \cdot \f \phi  \de x \Big|_0^T + \int_0^T \int_{\Td} \f v \cdot \t \f \phi  - \zeta ( \f F, \cof \f F , \det \f F ) : \nabla \f \phi\de x\de t \\
 + \int_0^T c \| \nabla \f \phi \|_{L^p(\Td;\R^{d\times d})} \left [ \E(\f v ,\f F)- E\right ] \de t & \leq 0 \,
\end{align}
for all 
$ \f \phi \in \C^1(\Td\times [0,T]; \R^d )$ and a.e.~$s<t\in (0,T)$. 
Defining $\mathcal{V}= \mathcal{U} := W^{1,1}(0,T; W^{1,p}(\Td))$ with 
\begin{align*}
\f l: \mathcal{U} \to \R \qquad \langle \f l , \f \varphi \rangle := - \int_{\Td} \f v \cdot \f \phi  \de x \Big|_0^T + \int_0^T \int_{\Td} \f v \cdot \t \f \phi  - \zeta ( \f F ,\cof \f F , \det \f F) : \nabla \f \phi\de x \de t 
\end{align*}
as well as 
\begin{align*}
\mathfrak{p}: L^1(0,T; L^p  ({\Td} ; \R^{d\times d}) \to [0,\infty) \qquad \mathfrak{p}(\f A):= c \int_0^T\| \f A \|_{L^p(\Td;\R^{d\times d})} \left [ E- \E(\f v ,\f F)\right ] \de t \,,
\end{align*}
we may apply Lemma~\ref{lem:hahn} in order to infer that there exists 
$\mathfrak R \in L^\infty_{w}(0,T;L^{p'}(\Td ; \R_{\sym}^{d\times d}
)) $
such that $\langle - \mathfrak{R}, \f A \rangle \leq \mathfrak{p}(\f A ) $ and $\langle \f l , \f \varphi \rangle = \langle - \mathfrak{R}, \f \varphi \rangle $. 
With the definition of $\f l$ and taking the supremum in $\f A$ , we find that
\begin{align}\label{lineqpoly}
- \int_{\Td} \f v \cdot \f \phi  \de x \Big|_0^T + \int_0^T \int_{\Td} \f v \cdot \t \f \phi  - \zeta( \f F ,\cof \f F , \det \f F)  : \nabla \phi\de x + \int_0^T \int_{{\Td}} \nabla \f \varphi :  \de \mathfrak{R}( x)  = 0 \,
\end{align}
as well as  $ \| \mathfrak{R}(t) \| _{L^{p'}(\Td; \R^{d\times d} )} \leq c \left [E(t) - \E(\f v(t) , \f F(t))\right ]$.
%
Now, we want to construct the Young measure $\nu$ via Lemma~\ref{lem:Young}.
%
Therefore, we define 
 $$g:\R^ d \times  \R^{d\times d} \times \R^{d\times d} \times \R  \to \R^{d\times d} \text{ via }  g ( \f s , \f S, \f R , r   ) := \zeta( \f S, \f R, R) $$ as well as $$\eta: \R^ d \times  \R^{d\times d} \times \R^{d\times d} \times \R  \to \R  \text{ via }\eta (\f s , \f S, \f R, r ) = \frac{1}{2}| \f s |^2 + G ( \f S , \f R, r) $$ and observe that $\eta$ fulfills all requirements of Lemma~\ref{lem:Young} thanks to Hypothesis~\ref{hyp:elast}. 
 The growth of $\eta$ allows to identify with $ \Omega = \Td$ that $ V = L^2(\Td; \R^d) \times L^p(\Td; \R^{d\times d} ) \times L^q(\Td;\R^{d\times d}) \times L^r(\Td; \R) $. We define $$ 
  W : = \{  \f A \in L^p(\Td; \R^{d\times d}) \mid \exists \f y \in W^{1,p}( \Td; \R^{d}) \text{ such that } \nabla \f y = \f A  \}$$ and   equip it with the norm  $\| \f S\|_{W}  =\frac{1}{c}\|\f S\|_{L^2(\Omega;\R^{d\times d})}  $. 
We want to show that $ \zeta $ induces a surjective mapping from $V$ to $W^*$. 
First, we show that the Nemytskii mapping  induced by $ \zeta $ is well defined. 
Therefore, we observe 
\begin{align*}
| \zeta (\f S, \f R , r) |^{\frac{p}{p-1}} &\leq C \left ( \left | \frac{\partial G}{\partial \f F}(\f S, \f R , r)\right |^{\frac{p}{p-1}}  + | \f S |^{\frac{p}{p-1}} \left |\frac{\partial G}{\partial \f Z}(\f S, \f R , r)\right |^{\frac{p}{p-1}} +  | \f S |^{\frac{2p}{p-1}} \left |\frac{\partial G}{\partial r}(\f S, \f R , r)\right |^{\frac{p}{p-1}}\right )
\\
&\leq C \left (| \f S|^p + \left | \frac{\partial G}{\partial \f F}(\f S, \f R , r)\right |^{\frac{p}{p-1}}  +  \left |\frac{\partial G}{\partial \f Z}(\f S, \f R , r)\right |^{\frac{p}{p-2}} +  \left |\frac{\partial G}{\partial r}(\f S, \f R , r)\right |^{\frac{p}{p-3}}\right )
\\
&\leq C \left (| \f S|^p +|\f R|^q + |r|^r +1 \right )\,.
\end{align*}
This implies that $\zeta $ induces a mapping $\zeta : V \to W^*$ that is well defined. Additionally, this implies that the estimate~\eqref{bound} is fulfilled. 
Now, we want to show that this mapping is also surjective. 
Therefore, we consider for every $ \f A \in W$ the functional $$ G_{\f A} : W^{1,p}(\Omega) \to \R \quad \text{via}\quad G _ {\f A}(\f y ) := \int_{\Td} G ( \nabla \f y , \cof \nabla \f y , \det  \nabla \f y ) - \f A : \nabla \f y \,.$$ This is a lower semicontinuous functional with $p$-growth according to Hypothesis~\ref{hyp:elast}, which guarantees the existence of minimizers (\textit{cf}.~\cite[Thm.~6.5]{rindler}).  Due to the regularity of $G$, we can derive the Euler--Lagrange equations, which are given by 
\begin{align*}
\int_{\Td} \zeta(\nabla \f y , \cof \nabla \f y , \det \nabla \f y ) : \nabla \f \varphi - \f A : \nabla \f \varphi \de x = 0 \quad \text{for all }\f \varphi \in W^{1,p}(\Td) \,.
\end{align*}
Since, $\f A$ was arbitrary, this guarantees the surjectivity of the mapping $ \nabla \f y \mapsto \zeta (\nabla \f y , \cof \nabla \f y ,\det \nabla \f y )$ from $W^{1,p}(\Td)$ onto $W$. This in turn also implies the surjectivity of the mapping $ \zeta$ as a mapping from $V $ to $W^*$. We note that the element $\f A$ is uniquely defined in the product with all gradients $\nabla \f \varphi \in L^p(\Td;\R^{d\times d})$ as it is curl-free due to the condition $\f A \in W$.

 We are now in the position to apply Lemma~\ref{lem:Young} for $ \f y_0 = (\f v , \f F, \f Z , w ) = ( \t \f y , \nabla \f y , \cof \nabla \f y , \det \f y)$, $ \f z _0 = \mathfrak{R}$ and $ x_0 = E-\E(\f v, F)$. 
  According to Lemma~\ref{lem:Young} there exists a $$ \nu \in L^\infty_{w^*}(0,T;L^1(\Td)\times \mathcal{M}(\R^d\times \R^{d\times d} \times \R^{d\times d} \times \R )) \text{ and }\gamma \in L^\infty_{w^*}(0,T;\mathcal{M}^+(\Td))$$  such that 
\begin{align*}
1 =  \langle  \nu_{x,t} , 1\rangle  \,, \quad 
(\f v(x,t), \f F(x,t), \f Z(x,t), w(x,t)) =\langle  \nu_{x,t} , (\f s , \f S, \f R, r )\rangle  \,, \quad \text{a.e.~in $\Omega \times (0,T)$ and } \\\quad  \zeta ( \f F (x,t), \f Z(x,t), r(x,t))  - \mathfrak{R}(x,t)= \langle \nu_{x,t}, \zeta(\f S, \f R, r )\rangle   \quad \text{in }W^*
\end{align*} 
a.e.~in $\Omega \times (0,T)$ as well as 
\begin{align*}
\E(\t\f y, \f y  ) + \frac{1}{c}\| \mathfrak{R}\|_{W^*}  + E-\E(\t \f y , \f y) = \int_{\Td} \left  \langle \nu , \frac{1}{2}|\f s |^2 + G( \f S, \f R, r ) \right \rangle  \de x + \int_{\Td} \de \gamma_t(x) \,.
\end{align*}
This inserted into~\eqref{lineqpoly} implies the formulation~\eqref{eqpoly:elastmeas3}. The monotonicity of $E$, stemming from choosing $\f \varphi \equiv \f 0$ as well as $\Psi  \equiv \f 0$ in~\eqref{envarin:polyelast},  then implies~\eqref{eqpoly:elastmeas3}.


This proves the first assertion of  Theorem~\ref{thm:mainelast2}.
We note that due to the estimate on~$\mathfrak{R}$, and its identification, we find that the estimate~\eqref{est:nuelast2} holds true.

Let $ (\f v, \f F,\f Z,w ,  \nu)$ be a measure-valued solution, where~\eqref{eq:polyelastmeas3} is replaced by: there exists $E:[0,T]\to \R$ non-increasing such that $E(0) = \E(\f v_0, \f F_0)$ and~\eqref{strongeninelast} as well as~\eqref{est:nuelast2} are fulfilled. 
First, we observe that we may identify $ ( \f v , \f F , \f Z , w) = ( \t \f y , \nabla \f y , \cof \nabla \f y, \det \nabla \f y ) $ according to~\eqref{eqpoly:inden} and the last point of Definition~\ref{def:polyelastmeas}. We note that the function $\f y$ could also be constructed by setting $ \f y := -(-\Delta)^{-1} ( \di \f F)$ in the weak sense, where $  (-\Delta)^{-1}$ denotes the inverse of the Laplace operator on the torus given by convolution with the appropriate Green's function. 
This is only unique, if the mean $\ov{\f y}= \int_{\Td} \f y (x,t)\de x $ is fixed, which should be fixed to $ \ov{\f y}(t)=\ov{\f y}_0 + t \ov{\f v}_0$. From~\eqref{eq:polyelastmeas2} we can identify $ \f v = \t \f y$ in the weak sense. 

Due to the convexity of $\eta$, we infer from Jensen's inequality
\begin{equation}
\begin{aligned}
E(t) &\geq \int_{\Td} \langle \nu_{x,t}, \eta ( \f s, \f  S, \f R, r ) \rangle \de x \\&\geq \int_{\Td} \eta \left ( \langle \nu_{x,t}, ( \f s, \f  S, \f R, r ) \rangle \right ) \de x\\&  = \int_{\Td} \eta ( \f v (x,t) , \f F(x,t) , \f Z(x,t), w(x,t))\de x  \\&=\int_{\Td} \eta ( \f v (x,t) , \f F(x,t) , \cof \f F(x,t), \det \f F(x,t))\de x\\ &  =\int_{\Td} \eta ( \t \f y (x,t) , \nabla \f y (x,t) , \cof \nabla \f y(x,t), \det \nabla \f y(x,t))\de x
\\
&=  \E(\t\f y (t),\f y(t) ) 
\end{aligned}
\end{equation}
for a.e.~$t\in(0,T)$
Now, we subtract from $ E(t) - E(s) \leq 0$ the four equations~\eqref{eq:polyelastmeas1}--\eqref{eq:polyelastmeas4} with $ ( \f v , \f F , \f Z , w) = ( \t \f y , \nabla \f y , \cof \nabla \f y, \det \nabla \f y ) $ and add as well as subtract the term $\zeta (\nabla \f y , \cof \nabla \f y , \det \nabla \f y )$

\begin{multline*}
\left  [ E- \int_{\Td} \t \f y  \cdot \f \varphi + \nabla \f y  : \Psi  + \operatorname{cof} \nabla \f y  :  \Xi+ \det \nabla \f y  \phi  \de x \right ] \Big|_s^t 
\\
+ \int_s^t \int_{\Td} \t \f y \cdot \t \f \varphi + \nabla \f y   : \t \Psi + \operatorname{cof} \nabla \f y  :\t   \Xi+ \det \nabla \f y   \t \phi \de x \de \tau 
\\
- \int_s^t \int_{\Td} \zeta ( \nabla \f y ,  \operatorname{cof} \nabla \f y , \det \nabla \f y  ) : \nabla \f \varphi +  \t \f y  \cdot \di \Psi  
\de x \de \tau 
\\
- \int_s^t \int_{\Td}
(\t \f y  \times \nabla \f y  )  : \curl  \Xi  
+(\operatorname{cof} \nabla \f y )^T \t \f y  \cdot \nabla \phi  
%
\de x \de \tau   \\- \int_s^t\int_{\Td} \left [ \zeta ( \nabla \f y ,  \operatorname{cof} \nabla \f y , \det \nabla \f y  )  - \langle \nu_{x,t}, \zeta(\f S, \f R, r )\rangle \right ] : \nabla \f \varphi  \de x \de \tau \leq 0 
\end{multline*}
for a.e.~$T\geq t > s \geq 0$.
The estimate~\eqref{est:nuelast2} now implies the energy-variational inequality~\eqref{envarin:polyelast}. 
\end{proof}

\section{Ericksen--Leslie equations equipped with the Oseen--Frank energy}\label{sec:EL}
The Ericksen–Leslie equations describe the dynamics of nematic liquid crystals, a state of matter that exhibits properties between conventional liquids and solid crystals.
The system couples the momentum balance equation for the velocity field $\f v$, which includes anisotropic effects, with a gradient flow equation for the elastic energy associated with the director field $\f d$. The director field models the locally preferred molecular alignment direction, representing either rod-like molecules suspended in a liquid solvent or the intrinsic structure of the liquid crystal itself. This coupling is highly nonlinear, reflecting the complex interplay between hydrodynamics and molecular orientation.
Let $\f v: \overline{\Omega} \times [0,T] \to \mathbb{R}^d$ denote the velocity of the fluid,  $p: \overline{\Omega} \times [0,T] \to \mathbb{R}$ its pressure and $\f d: \overline{\Omega} \times [0,T] \to \mathbb{R}^d$ the director for $d=2,3$.
We consider the system governed by the equations
\begin{subequations}\label{system_to_reformulate}
\begin{align}
\label{system_a_to_reformulate}
    \partial_t \f v + (\f v \cdot \nabla) \f v  + \nabla p + \nabla \cdot \f T^E - \nabla \cdot \f T^L &= \f g, && 
    \nabla \cdot \f v = 0,  
\\
\label{system_c_to_reformulate}
    \partial_t \f d + (\f v \cdot \nabla) \f d - (\nabla \f v)_{skw} \f d +(I - \f d \otimes \f d ) ( \lambda (\nabla \f v)_{sym} \f d + \f q) &=\f 0, 
&& 
    |\f d|=1 
\end{align}
\end{subequations}
in $\Omega \times (0,T)$,
where we employ the initial conditions
\begin{align*}
    \f v(0) = \f v_0, \quad \f  d(0) = \f d_0 \text{ with } |\f d_0| = 1 \text{ in } \Omega 
\end{align*}
and boundary conditions
\begin{align*}
    \f v = 0, \quad \f d =\f  d_{0}  \text{ with } |\f d_0| = 1 \text{ on } \partial \Omega \times (0,T) \, .
\end{align*}
The function $\f q$ in system \eqref{system_to_reformulate} is the variational or Gateaux derivative of the free energy functional $ {F}: \R^ {d\times d} \times \R^d \to \mathbb{R} $ with respect to $\f d$ via 
\begin{align*}
   \f q : = \frac{\partial F(\f d,\nabla \f d)}{\partial \f d}  - \nabla \cdot \left (\frac{\partial F(\f d,\nabla \f d)}{\partial \nabla \f d} \right )
\end{align*}
where $F$ is called the Helmholtz free energy potential.
Further the Ericksen stress tensor is defined as 
\begin{align*}
   \f  T^E \coloneqq
    \nabla \f d^T \frac{\partial}{\partial \nabla \f d} F(\f d,\nabla\f  d),
\end{align*}
and the Leslie stress tensor $T^L$ is defined by
\begin{align*}
  \f   T^L \coloneqq
    & (\mu_1 + \lambda^2) (\f d \cdot (\nabla \f v)_{sym} \f d) (\f d \otimes \f d)
    + \mu_4 (\nabla \f v)_{sym} \\&
    +(\mu_5 + \mu_6 - \lambda^2) (\f d \otimes (\nabla \f v)_{sym}  \f d)_{sym}
    \\
    &- \lambda [\f d \otimes(I- \f d \otimes \f d)  \cdot\f  q]_{sym}
    - [\f d \otimes\f  q]_{skw},
\end{align*}
with $ \mu_4 >0, \quad \mu_5 + \mu_6 - \lambda^2\geq 0, \quad \mu_1 +\lambda^2 \geq 0$ in order to ensure the dissipative character of the model.

The \textit{Oseen--Frank} energy is given by~(see~Leslie~\cite{leslie}) \
\begin{align*}
F(\f d , \nabla \f d) := \frac{K_1}{2} (\di \f d )^2 +\frac{K_2}{2}( \f d \cdot \curl \f d )^2  + \frac{K_3}{2} |\f d \times \curl \f d|^2 + \frac{K_4}{2}
\left(\tr ((\nabla \f d )^2)- (\di \f d)^2  \right) 
 \,,
\end{align*}
where $K_1,K_2,K_3, K_4>0$.
The last term is uniquely determined by the prescribed boundary conditions since it can be written as a boundary term via
\begin{align*}
\int_\Omega \left(\tr ((\nabla \f d )^2)- (\di \f d)^2  \right)  \de x = \int_\Omega \di \left( \nabla \f d \f d - (\di \f d) \f d \right)\de x = \int_{\partial \Omega} \f n \cdot \left( \nabla \f d _0\f d_0 - (\di \f d_0) \f d_0 \right) \de S \,.
\end{align*}
For simplicity, we want to consider the case of two different constants $ K_1 =K_2<K_3$ and $K_4=0$, which is a reasonable assumption~\cite[Sec.~3.1.3.]{gennes}. 
This energy can be reformulated using the norm  restriction, $|\f d|=1$, to
\begin{align*}
F(\f d , \nabla \f d) = \frac{K_1}{2}|\nabla \f d|^2 + \frac{K_3-K_1}{2}| (\f d \cdot \nabla )\f d |^2 \,.
\end{align*}
For this, we used several calculation rules: 
\begin{align*}
\begin{split}
        (\di \f d )^2 &+ ( \f d\cdot \curl \f d )^2 + | \f d\times \curl \f d |^2 + \tr ( \nabla \f d^2 ) - ( \di \f d )^2 \\
    &= \tr ( \nabla \f d^2 ) + | \f d |^2 | \curl \f d |^2 = \tr ( \nabla \f d^2 ) + 2 | ( \nabla \f d)_{\skw} |^2\\
   & = \tr ( \nabla \f d ^2 ) + \frac{1}{2}\tr ( ( \nabla \f d - \nabla \f d^T) ^T ( \nabla \f d - \nabla \f d^T)  )  = \tr ( \nabla \f d^T \nabla \f d) = | \nabla \f d |^2 \,,
\end{split}
\end{align*}
as well as 
$$ 
|\f d \times (\curl \f d)|^2 = 4 | ( \nabla \f d)_{\skw} \f d | ^2 = | ( \nabla\f d - \nabla \f d ^T) \f d |^2 = | ( \f d \cdot \nabla )\f d - \nabla \f d ^T \f d |^2 =  | ( \f d \cdot \nabla )\f d |^2\,,
$$ 
where we used the fact that  $ 2 \nabla \f d ^T \f d = \nabla | \f d|^2 = 0$ under  the norm  restriction, $|\f d|=1$. 
For simplicity, we set $ K_1 =K_2=1$ and $ K_3-K_1 = k $ such that we infer
\begin{align*}
\frac{\partial F}{\partial \nabla \f d }( \f d ,\nabla \f d) = \nabla \f d +k (\f d \cdot  \nabla) \f d \otimes \f d  \quad \text{and}\quad \frac{\partial F}{\partial  \f d }( \f d ,\nabla \f d) = k (\nabla \f d ^T \nabla \f d ) \f d \,.
\end{align*}
The full system thus specifies to~\eqref{system_a_to_reformulate} with 
\begin{align}\label{defqandTe}
\f q  = k (\nabla \f d ^T \nabla \f d ) \f d - \di \left(\nabla \f d +k (\f d \cdot  \nabla) \f d \otimes \f d \right) \quad \text{and}\quad \f T^E= \nabla \f d^T \nabla \f d \left(  I +k  \f d \otimes \f d \right) \,.
\end{align}

For this model, we can simplify the measure-valued solutions defined in~\cite{meas}.

\begin{definition}[Measure-valued solutions]\label{def:meas}
The tuple $( \f v ,\f d, \nu,  \mu , \kappa )$ consisting of the pair $(\f v , \f d)$ of velocity field $\f v$ and director field  $\f d$, the  Young measure $ \nu$ and the defect measures $(\mu , \kappa)$ (see below)  is  said to be a measure-valued solution to~\eqref{system_a_to_reformulate} with initial values $ ( \f v _0, \f d _0)$
\begin{itemize}
\item if
\begin{align}
\begin{split}
\f v &\in L^\infty(0,T;\Ha)\cap  L^2(0,T;\Hsig)
\cap W^{1,2}(0,T; ( \f W^{2,p}_{0,\sigma})^*) \text{ for }p>3,
\\ \f d& \in L^\infty(0,T;\He)\cap   W^{1,2}  (0,T;  \f  L^{\nicefrac{3}{2}} ),\\
\{\nu _{(\f x,t)}\}&  \subset \mathcal{P} ( \R^{3\times 3})\,, \text{ a.\,e.~in $\Omega\times (0,T)$} \, ,\\
 \{\mu _t \} &\subset \mathcal{M}^+(\ov\Omega; \R^{3\times 3})\,,\text{ a.\,e.~in $ (0,T)$} \, ,  \\
 \{\kappa_t\} &\subset \mathcal{M}^+(\ov\Omega; \R^3)\,,\text{ a.\,e.~in $ (0,T)$} 
 \, ,  
\end{split}\label{measreg}
\end{align}
\item and if
\begin{subequations}\label{meas}
\begin{align}
\int_0^T \int_\Omega \t \f v
\cdot  \f \varphi
+ (\f v\cdot \nabla) \f v\cdot  \f \varphi -  \ll{\nu_{x,t},\f S^T \f S ( I + k\f d \otimes \f d ):\nabla \f \varphi   }  \de x \de t &
\\-\int_0^T \ll{\mu_t,  \nabla\f \varphi }\de t + \intte{\int_\Omega \f T^L: \nabla \f \varphi \de x  } &={}\intte{ \left \langle \f g ,\f \varphi\right \rangle }\, ,
\label{eq:velo}
\end{align}
{as well as}
\begin{align}
\partial_t \f d + (\f v \cdot \nabla) \f d - (\nabla \f v)_{skw} \f d +(I - \f d \otimes \f d ) ( \lambda (\nabla \f v)_{sym} \f d + \f q) &=0 
\end{align}
{ almost everywhere in $\Omega \times (0,T)$ with }
\begin{align}
 \intte{\int_\Omega \ll{\nu_{x,t}, k(\f S^T \f S \f d )\cdot\f \psi   }\de x }
+ \intte{\int_\Omega (\nabla \f d + k \nabla\f d \f d \otimes \f d) :  \nabla \f \psi \de x }&
 \\-{}\intte{\left \langle\f q , \f \psi\right  \rangle
 - \langle \kappa _t , \f \psi \rangle } &=0
  \,, 
\label{eq:q}
\end{align}%
\end{subequations}
holds for all $ \f \varphi \in \mathcal{C}_c^\infty(\Omega \times ( 0,T);\R^3)$ with $ \di \f \varphi =0$
and $ \f \psi \in  \mathcal{C}_c^\infty(\Omega \times ( 0,T);\R^3)$, respectively,
\item  the energy inequality 
\begin{align}
\begin{split}
 &\frac{1}{2}\|\f v (t)\|_{\Le}^2 + \ll{\nu_t, F} + \frac{1}{c}\ll{\mu_t , I } + \frac{1}{c} \langle  \kappa _t, \f d \rangle    + \inttet{(\mu_1+\lambda^2)\|\f d\cdot \sy{v}\f d\|_{L^2}^2 }  \\
& +\inttet{ \left [
  \mu_4 \|\sy{v}\|_{\Le}^2+( \mu_5+\mu_6-\lambda^2)\|\sy{v}\f d\|_{\Le}^2  +  \|\f d \times \f q\|_{\Le}^2\right ]}
\\
& \qquad\qquad \leq  \left ( \frac{1}{2}\|\f v_0 \|_{\Le}^2 + \int_{\Omega} F ( \f d_0, \nabla \f d_0 )\de x \right )
 + \inttet{\left [\langle \f g , \f v \rangle 
 \right ]}
\end{split}
\label{energyin}
\end{align}
is fulfilled for a.e. in $(0,T)$ and a certain $c>0$,

\item 
additionally, the norm restriction of the director holds, \textit{i.e.,} $|\f d (\f x,t)|=1$ for a.\,e.~$(\f x, t)\in \Omega\times (0,T)$, the Young measure of a linear function is the gradient of the director
\begin{align}
  \int_{\R^{3\times 3} } \f S  \nu_{(\f x, t)} ( \de \f S)   = \nabla \f d(\f x,t) \, \quad \text{for a.e. } (\f x ,t) \in \Omega\times (0,T)\,,\label{identify}
\end{align}
and the
initial conditions $( \f v_0, \f d_0)\in \Ha \times \He$  shall be fulfilled in the weak sense and the boundary conditions$ (\f v , \f d)= ( \f 0 , \f d_0)  $ on $ \partial \Omega$ in the sense of the trace.
We remark that the trace is well defined for the function $\f d \in L^\infty(0,T;\He)$, which is the expected value of the  measure $\nu$. 

\end{itemize}

The dual pairings is defined by
\begin{align*}
\ll{\nu_t, f } :={}& 
 \int_{\Omega} \int_{\R^{3\times 3} } f(\f x, t,\f d(\f x, t), \f S)  \nu_{(\f x, t)} ( \de \f S)\de \f x 
\end{align*}
for $f \in L^1(0,T;\C(\Omega)\times \C(\R^{d\times d}))$ such that $ \lim _{|\f S |\to \infty}\frac{|f(x,t, \f S )}{1+|\f S|^2}< \infty$.
For the defect measure, we use the usual dual pairing for continuous functions 
\begin{align*}
\langle \mu_t , \f A \rangle = \int_{\ov\Omega} \f A ( \f x ) : \de \mu_t(
x) \quad \text{and}\quad \langle \kappa_t , \f a \rangle = \int_{\ov\Omega} \f a ( \f x )\cdot   \de \kappa_t(
x)
\end{align*}
for all $ \f A \in \C(\ov \Omega  ; \R^{d\times d})$ and $ \f a \in \C(\ov \Omega ; \R^d)$, where $\kappa$ is such that $ \langle \kappa _t,\f d\rangle \geq 0$. Note that this dual product is well-defined by the embedding $ \mathcal{M}(\Omega;\R^d) $ into $( L^\infty(\Omega; \R^d) )^*$. Meaning that there exists an extension of $ \kappa$ to a linear mapping on $L^\infty(\Omega; \R^d)$ fulfilling this property. 
\end{definition}

\begin{remark}
The Definition~\ref{def:meas} differs significantly form the definitin of measure-valued solutions in~\cite{meas}. We simplified the definition here for the sake of readability into this more modern version. The concentrations of the approximating sequence that are described by an concentration and a concentration-angle measure in~\cite{meas}, are only represented by a defect measure here. There is an additional defect measure in~\cite{meas}, which we include here in the defect measure $\mu$. 
Here we tried to simplify the notion of these defect measures in order to keep the expression of the defects readable. 
\end{remark}

We define $ \E : \Ha \times \He \to [0,\infty)$ via $ \E( \f v, \f d ):= \int_\Omega\frac{1}{2}| \f v |^2 + F(\f d , \nabla \f d ) \de x $.
\begin{definition}[Energy-variational solution] \label{def:envarlc}
A triple $ ( \f v , \f d , E) $ is called an energy-variational solution to~\eqref{system_a_to_reformulate} with initial values~$(\f v _0, \f d_0)$, if 
\begin{align*}
\f v &\in L^\infty(0,T;\Ha)\cap  L^2(0,T;\Hsig)
\cap W^{1,2}(0,T; ( \f W^{2,p}_{0,\sigma})^*) \text{ for }p>3,
\\
 \f d& \in L^\infty(0,T;\He)\cap   W^{1,2}  (0,T;  \f  L^{\nicefrac{3}{2}} 
 )
 \\
 E & \in \BV
 \end{align*}
with $ E \geq\E(\f v, \f d) $ for a.e.~$t\in (0,T)$ and there exists a $\f q\in L^\infty _{w^*} (0,T; H^{-1}(\Omega)\cap \mathcal{M}(\Omega))$ such that it holds  
\begin{multline}
\left  [
E - \int_\Omega \f v \cdot \f\varphi +  \f d \cdot  \f \zeta \de x  
\right ]\Big |_s^t + \int_s^t \int_\Omega \f v \cdot \t \f\varphi +  \f d \cdot  \t \f \zeta  +( \f v \otimes \f v) : \nabla \f \varphi + \nabla \f d ^T \nabla \f d ( I + k \f d \otimes \f d ) : \nabla \f \varphi    \de x 
\\+ \int_s^t\left [
(\mu_1+\lambda^2)\|\f d\cdot \sy{v}\f d\|_{L^2}^2  +   \mu_4 \|\sy{v}\|_{\Le}^2+( \mu_5+\mu_6-\lambda^2)\|\sy{v}\f d\|_{\Le}^2  \right ]\de \tau 
\\
+ \int_s^t   \|\f d \times \f q\|_{\Le}^2-\int_\Omega \f T^L : \nabla \f \varphi \de x + \langle \f g , \f \varphi - \f v \rangle  \de \tau 
\\
- \int_s^t \int_\Omega \left[  ( \f v \cdot \nabla )\f d- ( \nabla \f v )_{\skw} \f d  + ( I -\f d \otimes \f d) ( \lambda ( \nabla \f v )_{\sym} \f d + \f q) \right ] \cdot \f \zeta \de x \de \tau 
\\
- \int_s^t \langle \f q , \f \psi\rangle -\int_\Omega
k   (\nabla \f d^T \nabla \f d) \f d \cdot \f \psi +  \nabla \f d ( I + k \f d \otimes \f d ) : \nabla \f \psi \de x \de \tau 
\\
+ \int_s^t2 \left \| (\nabla \f \varphi + k \f d \otimes ( \nabla \f \varphi \f d + \f \psi ))_{\sym,-} \right \| _{L^\infty_{\f d} (\Omega;\R^{d\times d})} \left [\E(\f v ,\f d ) - E \right ]\de \tau \leq 0 \,.\label{envarineqlc}
\end{multline}
for all $ \f \varphi\in \C^1  ([0,T]; \C^1(\Omega ; \R^d))$ with $\di \f \varphi =0 $ in $\Omega \times (0,T)$, $\f \zeta\in \C^1([0,T];L^{3}(\Omega))$, $\f \psi  \in L^1(0,T;\C^1(\Omega; \R^d)) $, and for a.e. $0\leq s< t\leq T$ including $ s=0 $ with $ ( \f v (0), \f d (0))= ( \f v_0, \f d _0)$. 
\end{definition}
Here $ L^\infty_{\f d} (\Omega;\R^{d\times d})$ is the space of measurable essentially bounded functions from $ \Omega $ to $ \R^{d\times d}$, where we equip $ \R^{d\times d}$ with a norm $ | \cdot |_{\f d}$ that is equivalent to the usual spectral norm for $ \f d \in L^\infty ( \Omega ; \Se^{d-1})$ via 
\begin{equation}
\label{dnorm}
| \f A | ^2_{\f d} := \max\left \{ \inf _{\lambda \in \R} | \f A - \lambda \f d \otimes \f d |_2^2 , \frac{1}{k+1	} \f d \cdot \f A \f d\right \} \,,
\end{equation}
for $ \f A \in \R^{d\times d}$, where $|\cdot|_2$ denotes the usual spectral norm. Via some elementary calculations, we can identify the conjugate norm with respect to the Frobenius product, which is given by 
$ | \f A |_{\f d^*}^2 = \tr (\f A) + k \f d \cdot \f A \f d $ at least for symmetric positive semi-definite matrices $\f A\in \R^{d\times d}_{\sym,+}$.

\begin{definition}[Dissipative weak solutions]\label{def:dissweak}
A triple $ ( \f v , \f d , E) $ is called an dissipative-weak solutions with initial values~$(\f v _0, \f d_0)$, if 
\begin{align*}
\f v &\in L^\infty(0,T;\Ha)\cap  L^2(0,T;\Hsig)
\cap W^{1,2}(0,T; ( \f W^{2,p}_{0,\sigma})^*) \text{ for }p>3,
\\
 \f d& \in L^\infty(0,T;\He)\cap   W^{1,2}  (0,T;  \f  L^{\nicefrac{3}{2}} 
 )
 \\
 E & \in \BV
 \end{align*}
and there exists a defect measure $ \mathfrak{R}\in L^\infty_{w^*}(0,T;(L^\infty(\Omega;\R^{d\times d}_{\sym}))^*)$ such that $ \langle \mathfrak{R}, \f a \otimes \f a \rangle \geq 0 $ for all $\f a \in L^\infty(\Omega;\R^d)$ and 
\begin{subequations}\label{dissweak}
\begin{align}
\begin{split}
\intte{ \int_\Omega (\t \f v  +(\f v\cdot \nabla) \f v)\cdot  \f \varphi  - \nabla \f d ^T \nabla \f d  ( I + k\f d \otimes \f d ):\nabla \f \varphi \de x  }  &
\\- \intte{  \left  \langle \mathfrak R ,  \nabla \f \varphi( I + k\f d \otimes \f d ) \right \rangle_{L^\infty(\Omega;\R^{d\times d}_{\sym})}  -  \int_\Omega \f T^L : \nabla \f \varphi\de x  } &{}={}\intte{ \left \langle \f g ,\f \varphi\right \rangle }\, ,
\end{split}\label{eq:velodissweak}
\intertext{for all $ \varphi\in L^2 (0,T; \C^1(\Omega ; \R^d))$ with $\di \f \varphi =0 $ in $\Omega \times (0,T)$  as well as}
\partial_t \f d + (\f v \cdot \nabla) \f d - (\nabla \f v)_{skw} \f d +(I - \f d \otimes \f d ) ( \lambda (\nabla \f v)_{sym} \f d + \f q) &{}=0\label{eq:dirdissweak} 
\intertext{ almost everywhere in $\Omega \times (0,T)$ with }
\begin{split}
 \intte{ \int_\Omega k( \nabla \f d ^T \nabla \f d  \f d )\cdot\f \psi \de x   
+ \int_\Omega (\nabla \f d + k \nabla\f d \f d \otimes \f d) :  \nabla \f \psi \de x}& \\ +k\intte{ \left \langle  \mathfrak R ,  \f d \otimes \f \psi \right \rangle_{L^\infty(\Omega; \R^{d\times d}_{\sym})}  }
&{}={}\intte{\left \langle\f q , \f \psi\right \rangle}  \,, 
 \end{split}
\label{eq:qdissweak}
\end{align}%
for all  $\f \zeta \in L^1(0,T;\C^1(\Omega; \R^d)) $. 
Moreover, it holds that 
\begin{align}\label{eninls}
\begin{split}
E \Big |_s^t + \int_s^t\left [
(\mu_1+\lambda^2)\|\f d\cdot \sy{v}\f d\|_{L^2}^2  +   \mu_4 \|\sy{v}\|_{\Le}^2\right ]\de \tau 
\\+ \int_s^t\left [
( \mu_5+\mu_6-\lambda^2)\|\sy{v}\f d\|_{\Le}^2  +  \|\f d \times \f q\|_{\Le}^2\right ]\de \tau &
\leq \int_s^t \langle \f g ,  \f v \rangle \de \tau 
\end{split}
\end{align}
for almost every $ 0\leq s <t \leq T$ with 
\begin{align}\label{enindeissweak}
\E( \f v(t) , \f d(t) ) + \frac{1}{2}\langle \mathfrak{R(t)}, I + k \f d \otimes \f d \rangle_{L^\infty(\Omega;\R^{d\times d }_{\sym})} \leq E(t) \quad \text{for a.e.~}t\in (0,T).
\end{align}
\end{subequations}

\end{definition}
\begin{theorem}\label{thm:lcequivalence}
The tuble  $ (\f v ,\f d , E)$ is an energy-variatonal solution according to Definition~\ref{def:envarlc} if and only if it is a dissipative weak solution according to Definition~\ref{def:dissweak}. 

Let $ (\f v ,\f d , E)$ be an energy-variatonal solution according to Definition~\ref{def:envarlc} then there exists a measure-valued solution in the sense of~Definition~\ref{def:meas}. 
\end{theorem}
\begin{proof}[Proof of Theorem~\ref{thm:lcequivalence}]
Let $(\f v, \f d ,E)$ be an energy-variational solution in the sense of~Definition~\eqref{def:envarlc}. 
Choosing $ ( \f \varphi ,\f \psi, \f \zeta ) = ( 0 , 0 ,0)$ in~\eqref{envarineqlc} implies the inequality~\eqref{eninls}. 

Choosing $ ( \f \varphi ,\f \psi, \f \zeta ) = ( 0 , \frac{1}{\alpha }\f \xi  ,0)$ for $\alpha >0$ and $\f \xi \in L^2(0,T;L^{3}(\Omega))$ and multiplying the resulting inequality by $\alpha $ as well as considering $\alpha \searrow 0$, we infer the equation~\eqref{eq:dirdissweak} tested with~$\f \xi$, which is equivalent to~\eqref{eq:dirdissweak}. 

Choosing $ ( \f \varphi ,\f \psi, \f \zeta ) = (  \frac{1}{\alpha }\tilde{\f\varphi}   ,0, \frac{1}{\alpha }\tilde{\f \zeta}  )$ for $\alpha >0$ and 
$\tilde{\f \varphi}\in L^2 (0,T; \C^1(\Omega ; \R^d))$ with $\di \tilde{\f \varphi} =0 $ in $\Omega \times (0,T)$ as well as $\tilde{\f \zeta} \in L^1(0,T;\C^1(\Omega; \R^d)) $, and multiplying the resulting inequality by $\alpha $ as well as considering $\alpha \searrow 0$, we infer for $t=T$ and $s=0$, 
\begin{align}\label{lcineqhb}
\begin{split}
&\langle \f l , (\tilde{\f \varphi}, \tilde{\f \zeta} ) \rangle _{L^2 (0,T; \C^1(\Omega ; \R^d)) \times L^1(0,T;\C^1(\Omega; \R^d)) } 
\\
&{}= 
 - \int_\Omega \f v \cdot \tilde{\f \varphi} \de x \Big|_0^T + \int_0^T \int_\Omega \f v \cdot \t \tilde{\f \varphi} + ( \f v \otimes \f v + \nabla \f d ^T \nabla \f d (I+k\f d \otimes \f d ) - \f T^L):\nabla \tilde{\f \varphi} \de x + \langle \f g , \tilde{\f \varphi}\rangle  \de t 
\\
&\quad+ \int_0^T \int_\Omega k (\nabla \f d ^T\nabla \f d ) \f d \cdot \tilde{\f \zeta} + \nabla \f d ( I + k \f d \otimes \f d ) : \nabla \tilde{\f \zeta} \de x - \langle \f q , \tilde{\f \zeta}\rangle \de t 
\\
&{}\leq {}  \int_0^T \left \| (\nabla \tilde{\f \varphi} + k \f d \otimes ( \nabla \tilde{\f \varphi} \f d + \tilde{\f \zeta} ))_{\sym,-} \right \| _{L^\infty _{\f d}(\Omega;\R^{d\times d})} \left [E - \E(\f v ,\f d )  \right ]\de t 
\\
&= \int_0^T \| (A ( \tilde{\f \varphi}, \tilde{\f \zeta}))_{-} \|_{L^\infty_{\f d}(\Omega; \R^{d\times d})}\de t
\end{split}
\end{align}
with $ \f l \in \mathbb X^*$, where $$\mathbb X := (L^2(0,T; \C^1(\Omega ; \R^d)) \times L^1(0,T;\C^1(\Omega; \R^d)) )^*\quad \text{and}\quad  \mathbb Y =  L^1(0,T; L^\infty(\Omega; \R^{d\times d}_{\sym}))$$  such that  $$ A : \mathbb X \to \mathbb Y\quad\text{and}\quad A  ( \tilde{\f \varphi}, \tilde{\f \zeta}) =
 (\nabla \tilde{\f \varphi} + k \f d \otimes ( \nabla \tilde{\f \varphi} \f d + \tilde{\f \zeta} ))_{\sym}\,.$$
From~\eqref{lcineqhb}, we infer that $ \langle \f l ,x \rangle \leq p(A(x))$ for all $x\in \mathbb X$ with $$ p : \mathbb Y \to [0,\infty) \quad \text {via}\quad p(y)=\int_0^T 2\| (y)_{\sym,-}\|_{L^\infty_{\f d}(\Omega;\R^{d\times  d}_{\sym})}[E-\E(\f v, \f d)]\,.$$ According to Lemma~\ref{lem:fact}, we infer that there exists $\mathfrak{R} \in \mathbb Y^*$ such that $ \langle \f l ,x \rangle _{\mathbb X} = -\langle \mathfrak{R} , A(x)\rangle_{\mathbb{Y}} $ for all $ x\in \mathbb X$ and
$-\int_0^T \langle \mathfrak{R} , y\rangle_{\mathbb{Y}} \de t  \leq 
\int_0^T 2 \|(y)_{\sym,-}\|_{L^\infty(\Omega; \R^{d\times d})}$ for all $y\in\mathbb{Y}$. We can identify $ \mathbb Y^* = L^\infty_{w^*} (0,T; (L^\infty(\Omega; \R^{d\times d} _{\sym}))^*)$. 

 Choosing $ y = -\frac{1}{2}\phi(t) (I+ k \f d \otimes \f d )$ with $\phi \in \C([0,T])$, where $ \phi \geq 0$, we find 
 \[  
\int_0^T \phi  \frac{1}{2} \langle \mathfrak R, I +k \f d \otimes \f d \rangle_{\mathbb Y}\de t  \leq 
 \int_0^T \phi [ E - \E(\f v, \f d )] \de t \quad \text{for all }\phi \in \C([0,T]) \text{ with }\phi \geq 0\,.
 \]
This implies~\eqref{enindeissweak}. 

We note that we equipped $\R^{d\times d}_{\sym}$ with the special norm $|\cdot |_{\f d}$ defined above for which we observe   $ \| I + k\f d \otimes \f d  \| _{L^\infty_{\f d}(\Omega;\R^{d\times d}_{\sym})}=1$. 
Moreover, we observe that $ -\int_0^T \langle \mathfrak{R} , y\rangle_{\mathbb{Y}} \de t  \leq  0
$ for all $ y \in L^1(0,T; L^\infty(\Omega; \R^{d\times d}_{\sym,+}))$, which is nothing else than~$ \mathfrak{R}\in  L^\infty_{w^*} (0,T; (L^\infty(\Omega; \R^{d\times d} _{\sym_+}))^*)$.

Since an element $x$ of the space $\mathbb X$ has two components such that $ x=(\f \varphi , \f \zeta)$, 
the equality $ \langle \f l ,(\f \varphi , \f 0)  \rangle _{\mathbb X} = \langle \mathfrak{R} , (\f \varphi , \f 0)\rangle_{\mathbb{Y}}$ gives exactly~\eqref{eq:velodissweak} and the equality  $ \langle \f l ,(\f 0 ,\f \zeta)  \rangle _{\mathbb X} = \langle \mathfrak{R} , (\f 0 , \f \zeta)\rangle_{\mathbb{Y}}$ gives exactly~\eqref{eq:qdissweak}. 
This proves the first direction of the equivalence.

Now let $(\f v, \f d , E)$ be a dissipative weak solution in the sense of
Definition~\ref{def:dissweak}. Choosing $\f \varphi = \chi_{[s,t]}\tilde{\f\varphi}$ in~\eqref{eq:velodissweak} and adding~\eqref{eq:dirdissweak} multiplied by $\chi_{[s,t]}\f \psi$ and integrated in space and time, we infer after an integration by parts in time, that
\begin{multline}
\left  [
- \int_\Omega \f v \cdot \f\varphi +  \f d : \f \psi \de x  
\right ]\Bigg |_s^t \\+ \int_s^t \int_\Omega \f v \cdot \t \f\varphi +  \f d \cdot  \t \f \psi  +( \f v \otimes \f v) : \nabla \f \varphi + \nabla \f d ^T \nabla \f d ( I + k \f d \otimes \f d ) : \nabla \f \varphi    \de x \de \tau \\
- \int_s^t \int_\Omega \f T^L : \nabla \f \varphi \de x - \langle \f g , \f \varphi - \f v \rangle\de \tau 
+ \int_s^t \left \langle \mathfrak R,  (\nabla \f \varphi + k \f d \otimes ( \nabla \f \varphi \f d + \f \zeta ))_{\sym,-} \right \rangle _{L^\infty (\Omega;\R^{d\times d})}\de \tau
\\
 - \int_s^t\int_\Omega \left[  ( \f v \cdot \nabla )\f d- ( \nabla \f v )_{\skw} \f d  + ( I -\f d \otimes \f d) ( \lambda ( \nabla \f v )_{\sym} \f d + \f q) \right ] \cdot \f \psi \de x \de \tau 
 \leq 0 \,.\label{addedvanddir}
\end{multline}
 Now we subtract this from~\eqref{eninls} and also subtract~\eqref{eq:qdissweak} integrated only over the interval $[s,t]$, we infer
 \begin{multline*}
\left  [
E - \int_\Omega \f v \cdot \f\varphi +  \f d : \f \psi \de x  
\right ]\Bigg |_s^t \\+ \int_s^t \int_\Omega \f v \cdot \t \f\varphi +  \f d \cdot  \t \f \psi  +( \f v \otimes \f v) : \nabla \f \varphi + \nabla \f d ^T \nabla \f d ( I + k \f d \otimes \f d ) : \nabla \f \varphi    \de x 
\\+ \int_s^t\left [
(\mu_1+\lambda^2)\|\f d\cdot \sy{v}\f d\|_{L^2}^2  +   \mu_4 \|\sy{v}\|_{\Le}^2+( \mu_5+\mu_6-\lambda^2)\|\sy{v}\f d\|_{\Le}^2  \right ]\de \tau 
\\
+ \int_s^t \int_\Omega   \|\f d \times \f q\|_{\Le}^2 - \f T^L : \nabla \f \varphi \de x + \langle \f g , \f \varphi - \f v \rangle \de \tau 
\\
- \int_s^t \int_\Omega  \int_\Omega \left[  ( \f v \cdot \nabla )\f d- ( \nabla \f v )_{\skw} \f d  + ( I -\f d \otimes \f d) ( \lambda ( \nabla \f v )_{\sym} \f d + \f q) \right ] \cdot \f \psi \de x \de \tau 
\\
- \int_s^t \langle \f q , \f \zeta\rangle -\int_\Omega
k   (\nabla \f d^T \nabla \f d) \f d \cdot \f \zeta +  \nabla \f d ( I + k \f d \otimes \f d ) : \nabla \f \zeta \de x \de \tau 
\\
+ \int_s^t \left \langle \mathfrak R,  (\nabla \f \varphi + k \f d \otimes ( \nabla \f \varphi \f d + \f \zeta ))_{\sym,-} \right \rangle _{L^\infty (\Omega;\R^{d\times d})}\de \tau \leq 0 \,.
\end{multline*}
for all $ \f \varphi\in \C  ([0,T]; \C^1(\Omega ; \R^d))$ with $\di \f \varphi =0 $ in $\Omega \times (0,T)$, $\f \psi\in \C([0,T];L^{3/2}(\Omega))$, $\f \zeta \in L^1(0,T;\C^1(\Omega; \R^d)) $. The estimate~\eqref{enindeissweak} 
via 
\begin{align*}
&\left \langle \mathfrak R,  (\nabla \f \varphi + k \f d \otimes ( \nabla \f \varphi \f d + \f \zeta ))_{\sym,-} \right \rangle _{L^\infty (\Omega;\R^{d\times d})}\\&\qquad \leq  
\| \mathfrak R \|_{(L^\infty_{\f d} (\Omega;\R^{d\times d} _{\sym}))^*} 
\left \| (\nabla \f \varphi + k \f d \otimes ( \nabla \f \varphi \f d + \f \zeta ))_{\sym,-} \right \| _{L^\infty (\Omega)} 
\\
&\qquad\leq  \left [\E(\f v ,\f d ) - E \right ]\left \| (\nabla \f \varphi + k \f d \otimes ( \nabla \f \varphi \f d + \f \zeta ))_{\sym,-} \right \| _{L^\infty (\Omega)} 
\end{align*}
now implies~\eqref{envarineqlc}, where we used that $ \| \mathfrak R \|_{(L^\infty_{\f d} (\Omega;\R^{d\times d} _{\sym}))^*} =\langle \mathfrak R , I+ k \f d \otimes \f d  \rangle _{L^\infty (\Omega;\R^{d\times d} _{\sym})}$. 
This follows from the derivation of the dual norm of $| \cdot|_{\f d}$ below~\eqref{dnorm}
This proves the second direction of the equivalence. 
  
  Now we prove that the existence of a dissipative solution in the sense of Definition~\ref{def:dissweak} implies the existence of a measure-valued solution in the sense of Definition~\ref{def:meas}.
For the element $ \mathfrak{R}\in (L^\infty(\Omega;\R^{d\times d}_{\sym}))^*  $ such that $ \langle \mathfrak{R}, \f a \otimes \f a \rangle \geq 0 $ for all $\f a \in L^\infty(\Omega;\R^d)$, we apply the Yosida-Hewitt Decomposition~\cite{Yoshida} in order to decompose $ \mathfrak R$ into a Radon measure $ \mathrm{R}^{\mathrm{reg}}$ and a singular set function $ \mathrm{R}^{\mathrm{sing}}$ such that $ \mathfrak{R}= \mathrm{R}^{\mathrm{reg}}+\mathrm{R}^{\mathrm{sing}}$  with $ \langle\mathrm{R}^{\mathrm{sing}}, \f \varphi\rangle = 0 $ for all $\f \varphi \in \C(\Omega; \R^{d\times d}_{\sym})$. Moreover, applying the Radon--Nikkodym derivative to the regular part, we find that
$ \\de mathrm{R}^{\mathrm{reg}} (x) = B(x) \de x  + \de \mathfrak{R}^d(x) $ with $B\in L^1(\Omega; \R^{d\times d}_{\sym,+})$, and $\de x$ denoting the Lebesgue measure and $\mathfrak{R}^d$ a singular measure. 

For the part that is absolutely continuous with respect to the Lebesgue-measure, we observe 
that there exists a measure $\nu$ such that 
\begin{align}\label{expvari}
\begin{split}
1 = \int_{\R^{d\times d}} \de \nu_{x,t}  \,, \quad \nabla \f d (x,t) = \int_{\R^d\times d} \f S \nu_{x,t} (\de \f S) \,, \\ \nabla \f d (x,t)^T \nabla \f d (x,t) + B (x,t) = \int_{\R^{d\times d}} \f S^T \f S \nu_{x,t}(\de \f S)\,.
\end{split}
\end{align}
  There are many possibilities to construct such a measure, a similar way as in~Lemma~\ref{lem:Young} would be possible. We can also construct a normal distribution via 
   \[
\nu_{x,t}(\f S ) :=\frac{1}{(2 \pi)^{d^2/2} \det( B{(x,t)})^{d/2}} \exp\left( -\frac{1}{2} (\f S - \f A (x,t)):  ( B{(x,t)})^{-1} (\f S - \f A (x,t)) \right) \de \f S  \,.
\]
  This is one possible choice fulfilling~\eqref{expvari} but there are many others. 
  Moreover, we observe that $2 \langle \nu , F(\f d , \f S) \rangle = \langle \nu , \f S^T \f S :(I + k\f d \otimes\f d )\rangle = (\nabla \f d ^T \nabla \f d  + B) :( I + k \f d \otimes \f d )$.

We consider the linear mapping $ \f \zeta \mapsto k\left \langle \mathfrak R^{\mathrm{sig}}+\mathfrak{R}^d , \f d \otimes \f \zeta \right \rangle_{L^\infty(\Omega;\R^{d\times d}_{\sym})} $ it is a linear continuous mapping on $ L^1(0,T;\C(\Omega;\R^d))$ such that there exists a measure $ \kappa \in L^\infty_{w^*}(0,T;\mathcal{M}(\Omega;\R^d))$ such that 
\begin{equation}
\langle \mathfrak R^{\mathrm{sig}}+\mathfrak{R}^d, \f d \otimes \f \zeta\rangle _{L^\infty(\Omega;\R^{d\times d}_{\sym})} = \int_{\Omega}\f \zeta \cdot\kappa (\de x)\quad \text{for all }\f \zeta \in  L^1(0,T;\C(\Omega;\R^d))\,.
\label{kappa}
\end{equation}
Furthermore, the mapping $ A \mapsto  \left \langle \mathfrak R^{\mathrm{sig}}+\mathfrak{R}^d ,(I + k  \f d \otimes \f d) A  \right \rangle_{L^\infty(\Omega;\R^{d\times d}_{\sym})} $ it is a linear continuous mapping on $ L^1(0,T;\C(\Omega;\R^{d\times d}))$ such that there exists a measure $ \mu \in L^\infty_{w^*}(0,T;\mathcal{M}(\Omega;\R^{d\times d}))$ such that 
\begin{equation}
\langle \mathfrak R^{\mathrm{sig}}+\mathfrak{R}^d,(I + k  \f d \otimes \f d) A \rangle _{L^\infty(\Omega;\R^{d\times d}_{\sym})} = \int_{\Omega}A : \mu (\de x)\quad \text{for all }A \in  L^1(0,T;\C(\Omega;\R^{d\times d}))\,.
\label{mu}
\end{equation}
  We note that these measures are singular, \textit{i.e.,} concentrated on a set of Lebesgue measure-zero in $\Omega$. 
  To see this, we refer to~\cite[Thm.~III, 7.8]{dunfordschwarz}, since $\mathfrak R^{\mathrm{sig}}$ is a purely finitely additive set function and $\mathfrak R^d$ is concentrated on a set of Lebesgue-measure zero. 
  Using this definitions, we infer from~\eqref{eq:velodissweak} the formulation~\eqref{eq:velo} and from~\eqref{eq:qdissweak} the formulation~\eqref{eq:q}, respectively. 

  Extending the linear function $\kappa $ and $\mu$ to functionals on $L^\infty (\Omega)$ again, we infer
  that
\begin{align*}
   \langle \kappa , \f d \rangle &= \langle \mathfrak R^{\mathrm{sig}}+\mathfrak{R}^d, \f d \otimes \f d\rangle _{L^\infty(\Omega;\R^{d\times d}_{\sym})} \geq 0 \text{ and }\\  \langle \mu , I \rangle & =\langle \mathfrak R^{\mathrm{sig}}+\mathfrak{R}^d,(I + k  \f d \otimes \f d)  \rangle _{L^\infty(\Omega;\R^{d\times d}_{\sym})}   \geq 0\,.
\end{align*}  
  This especially implies  that there exists a $ c>0$ such that 
  \begin{align*}
  \langle \kappa, \f d \rangle + \langle \mu , I \rangle  \leq   c \langle \mathfrak R^{\mathrm{sig}}+\mathfrak{R}^d,(I + k  \f d \otimes \f d)  \rangle _{L^\infty(\Omega;\R^{d\times d}_{\sym})} = c \| \mathfrak R^{\mathrm{sig}}+\mathfrak{R}^d \|_{L^\infty_{\f d}(\Omega;\R^{d\times d}_{\sym})}\,. 
  \end{align*}
  From this, we infer 
  \begin{align*}
  &\E( \f v , \f d ) + \frac{1}{2}\langle \mathfrak{R}, I + k \f d \otimes \f d \rangle _{L^\infty(\Omega;\R^{d\times d }_{\sym})} \\&\quad =  \frac{1}{2} \int_{\Omega} | \f v |^2 +  \langle \nu , F(\f d , \f S)\rangle\de x  +\langle \mathfrak R^{\mathrm{sig}}+\mathfrak{R}^d,(I + k  \f d \otimes \f d)  \rangle _{L^\infty(\Omega;\R^{d\times d}_{\sym})} 
  \\& \quad \geq  \frac{1}{2} \int_{\Omega} | \f v |^2 +  \langle \nu , F(\f d , \f S)\rangle\de x + \frac{1}{c} \left (\langle \kappa, \f d \rangle + \langle \mu , I \rangle \right ) \,,
  \end{align*}
  which implies that~\eqref{energyin} holds true such that the assertion is proven. 
  \end{proof}
\section*{Acknowledgment} 
  The author acknowledges funding by the Deutsche Forschungsgemeinschaft (DFG, German Research Foundation) within SPP 2410 Hyperbolic Balance Laws in Fluid Mechanics: Complexity, Scales, Randomness (CoScaRa), project number 525941602.

  \end{document}